\documentclass[copyright,creativecommons]{eptcs}
\providecommand{\event}{Linearity \& TLLA 2020} 
\usepackage{underscore}           
\newif\ifeptcs\eptcstrue
\newif\ifcref\creftrue
\usepackage{amssymb,amsmath,stmaryrd,mathrsfs}

\def\definetac{\newif\iftac}    
\ifx\tactrue\undefined
  \definetac
  \ifx\state\undefined\tacfalse\else\tactrue\fi
\fi

\def\definebeamer{\newif\ifbeamer}
\ifx\beamertrue\undefined
  \definebeamer
  \ifx\uncover\undefined\beamerfalse\else\beamertrue\fi
\fi

\def\definecref{\newif\ifcref}
\ifx\creftrue\undefined
  \definecref
  \creffalse
\fi

\iftac\else\usepackage{amsthm}\fi
\usepackage{tikz,tikz-cd}
\usetikzlibrary{arrows}
\ifbeamer\else
  \usepackage{enumitem}
  \usepackage{xcolor}
  \definecolor{darkgreen}{rgb}{0,0.45,0} 
  \iftac\else\ifeptcs\else\usepackage[pagebackref,colorlinks,citecolor=darkgreen,linkcolor=darkgreen]{hyperref}
  \renewcommand*{\backref}[1]{}
  \renewcommand*{\backrefalt}[4]{({%
      \ifcase #1 Not cited.%
            \or Cited on p.~#2%
            \else Cited on pp.~#2%
      \fi%
    })}\fi\fi
\fi
\usepackage{mathtools}          
\usepackage{ifmtarg}            
\usepackage{microtype}
\usepackage{braket}             

\usepackage{url}                
\usepackage{xspace}             
\ifcref\usepackage[nosort]{cleveref}\usepackage{aliascnt}\fi
\usepackage[status=draft,author=]{fixme}
\fxusetheme{color}
\usepackage{mathpartir}
\usepackage{cmll}

\makeatletter
\let\ea\expandafter

\def\mdef#1#2{\ea\ea\ea\gdef\ea\ea\noexpand#1\ea{\ea\ensuremath\ea{#2}\xspace}}
\def\alwaysmath#1{\ea\ea\ea\global\ea\ea\ea\let\ea\ea\csname your@#1\endcsname\csname #1\endcsname
  \ea\def\csname #1\endcsname{\ensuremath{\csname your@#1\endcsname}\xspace}}

\newcount\foreachcount

\def\foreachletter#1#2#3{\foreachcount=#1
  \ea\loop\ea\ea\ea#3\@alph\foreachcount
  \advance\foreachcount by 1
  \ifnum\foreachcount<#2\repeat}

\def\foreachLetter#1#2#3{\foreachcount=#1
  \ea\loop\ea\ea\ea#3\@Alph\foreachcount
  \advance\foreachcount by 1
  \ifnum\foreachcount<#2\repeat}

\def\definescr#1{\ea\gdef\csname s#1\endcsname{\ensuremath{\mathscr{#1}}\xspace}}
\foreachLetter{1}{27}{\definescr}
\def\definecal#1{\ea\gdef\csname c#1\endcsname{\ensuremath{\mathcal{#1}}\xspace}}
\foreachLetter{1}{27}{\definecal}
\def\definebold#1{\ea\gdef\csname b#1\endcsname{\ensuremath{\mathbf{#1}}\xspace}}
\foreachLetter{1}{27}{\definebold}
\def\definebb#1{\ea\gdef\csname d#1\endcsname{\ensuremath{\mathbb{#1}}\xspace}}
\foreachLetter{1}{27}{\definebb}
\def\definefrak#1{\ea\gdef\csname f#1\endcsname{\ensuremath{\mathfrak{#1}}\xspace}}
\foreachletter{1}{9}{\definefrak}
\foreachletter{10}{27}{\definefrak}
\foreachLetter{1}{27}{\definefrak}
\def\definesf#1{\ea\gdef\csname i#1\endcsname{\ensuremath{\mathsf{#1}}\xspace}}
\foreachletter{1}{6}{\definesf}
\foreachletter{7}{14}{\definesf}
\foreachletter{15}{20}{\definesf}
\foreachletter{21}{27}{\definesf}
\foreachLetter{1}{27}{\definesf}
\def\definebar#1{\ea\gdef\csname #1bar\endcsname{\ensuremath{\overline{#1}}\xspace}}
\foreachLetter{1}{27}{\definebar}
\foreachletter{1}{8}{\definebar} 
\foreachletter{9}{15}{\definebar} 
\foreachletter{16}{27}{\definebar}
\def\definetil#1{\ea\gdef\csname #1til\endcsname{\ensuremath{\widetilde{#1}}\xspace}}
\foreachLetter{1}{27}{\definetil}
\foreachletter{1}{27}{\definetil}
\def\definehat#1{\ea\gdef\csname #1hat\endcsname{\ensuremath{\widehat{#1}}\xspace}}
\foreachLetter{1}{27}{\definehat}
\foreachletter{1}{27}{\definehat}
\def\definechk#1{\ea\gdef\csname #1chk\endcsname{\ensuremath{\widecheck{#1}}\xspace}}
\foreachLetter{1}{27}{\definechk}
\foreachletter{1}{27}{\definechk}
\def\defineul#1{\ea\gdef\csname u#1\endcsname{\ensuremath{\underline{#1}}\xspace}}
\foreachLetter{1}{27}{\defineul}
\foreachletter{1}{27}{\defineul}

\def\autofmt@n#1\autofmt@end{\mathrm{#1}}
\def\autofmt@b#1\autofmt@end{\mathbf{#1}}
\def\autofmt@d#1#2\autofmt@end{\mathbb{#1}\mathsf{#2}}
\def\autofmt@c#1#2\autofmt@end{\mathcal{#1}\mathit{#2}}
\def\autofmt@s#1#2\autofmt@end{\mathscr{#1}\mathit{#2}}
\def\autofmt@f#1\autofmt@end{\mathsf{#1}}
\def\autofmt@k#1\autofmt@end{\mathfrak{#1}}
\def\autofmt@u#1\autofmt@end{\underline{\smash{\mathsf{#1}}}}
\def\autofmt@U#1\autofmt@end{\underline{\underline{\smash{\mathsf{#1}}}}}
\def\autofmt@h#1\autofmt@end{\widehat{#1}}
\def\autofmt@r#1\autofmt@end{\overline{#1}}
\def\autofmt@t#1\autofmt@end{\widetilde{#1}}
\def\autofmt@k#1\autofmt@end{\check{#1}}

\def\auto@drop#1{}
\def\autodef#1{\ea\ea\ea\@autodef\ea\ea\ea#1\ea\auto@drop\string#1\autodef@end}
\def\@autodef#1#2#3\autodef@end{%
  \ea\def\ea#1\ea{\ea\ensuremath\ea{\csname autofmt@#2\endcsname#3\autofmt@end}\xspace}}
\def\autodefs@end{blarg!}
\def\autodefs#1{\@autodefs#1\autodefs@end}
\def\@autodefs#1{\ifx#1\autodefs@end%
  \def\autodefs@next{}%
  \else%
  \def\autodefs@next{\autodef#1\@autodefs}%
  \fi\autodefs@next}

\DeclareSymbolFont{bbold}{U}{bbold}{m}{n}
\DeclareSymbolFontAlphabet{\mathbbb}{bbold}

\mdef\delbar{\overline{\partial}}

\newcommand{\dd}[1]{\ensuremath{\frac{\partial}{\partial {#1}}}}

\mdef\hf{\textstyle\frac12 }
\mdef\thrd{\textstyle\frac13 }
\mdef\qtr{\textstyle\frac14 }

\newcommand{\op}{^{\mathrm{op}}}

\mdef\Id{\mathrm{Id}}
\mdef\id{\mathrm{id}}
\alwaysmath{ell}
\alwaysmath{infty}

\alwaysmath{odot}
\def\frc#1/#2.{\frac{#1}{#2}}   
\mdef\ten{\mathrel{\otimes}}

\mdef\sqten{\mathrel{\boxtimes}}

\DeclareFontFamily{U}{min}{}
\DeclareFontShape{U}{min}{m}{n}{<-> udmj30}{}
\newcommand{\yon}{\!\text{\usefont{U}{min}{m}{n}\symbol{'210}}\!}

\DeclareFontFamily{U}{mathb}{\hyphenchar\font45}
\DeclareFontShape{U}{mathb}{m}{n}{
      <5> <6> <7> <8> <9> <10> gen * mathb
      <10.95> mathb10 <12> <14.4> <17.28> <20.74> <24.88> mathb12
      }{}
\DeclareSymbolFont{mathb}{U}{mathb}{m}{n}
\DeclareFontSubstitution{U}{mathb}{m}{n}
\DeclareFontFamily{U}{mathx}{\hyphenchar\font45}
\DeclareFontShape{U}{mathx}{m}{n}{
      <5> <6> <7> <8> <9> <10>
      <10.95> <12> <14.4> <17.28> <20.74> <24.88>
      mathx10
      }{}
\DeclareSymbolFont{mathx}{U}{mathx}{m}{n}
\DeclareFontSubstitution{U}{mathx}{m}{n}
\DeclareMathSymbol{\dotplus}       {2}{mathb}{"00}
\DeclareMathSymbol{\dotdiv}        {2}{mathb}{"01}
\DeclareMathSymbol{\dottimes}      {2}{mathb}{"02}
\DeclareMathSymbol{\divdot}        {2}{mathb}{"03}
\DeclareMathSymbol{\udot}          {2}{mathb}{"04}
\DeclareMathSymbol{\square}        {2}{mathb}{"05}
\DeclareMathSymbol{\Asterisk}      {2}{mathb}{"06}
\DeclareMathSymbol{\bigast}        {1}{mathb}{"06}
\DeclareMathSymbol{\coAsterisk}    {2}{mathb}{"07}
\DeclareMathSymbol{\bigcoast}      {1}{mathb}{"07}
\DeclareMathSymbol{\circplus}      {2}{mathb}{"08}
\DeclareMathSymbol{\pluscirc}      {2}{mathb}{"09}
\DeclareMathSymbol{\convolution}   {2}{mathb}{"0A}
\DeclareMathSymbol{\divideontimes} {2}{mathb}{"0B}
\DeclareMathSymbol{\blackdiamond}  {2}{mathb}{"0C}
\DeclareMathSymbol{\sqbullet}      {2}{mathb}{"0D}
\DeclareMathSymbol{\bigstar}       {2}{mathb}{"0E}
\DeclareMathSymbol{\bigvarstar}    {2}{mathb}{"0F}
\DeclareMathSymbol{\corresponds}   {3}{mathb}{"1D}
\DeclareMathSymbol{\boxleft}      {2}{mathb}{"68}
\DeclareMathSymbol{\boxright}     {2}{mathb}{"69}
\DeclareMathSymbol{\boxtop}       {2}{mathb}{"6A}
\DeclareMathSymbol{\boxbot}       {2}{mathb}{"6B}
\DeclareMathSymbol{\updownarrows}          {3}{mathb}{"D6}
\DeclareMathSymbol{\downuparrows}          {3}{mathb}{"D7}
\DeclareMathSymbol{\Lsh}                   {3}{mathb}{"E8}
\DeclareMathSymbol{\Rsh}                   {3}{mathb}{"E9}
\DeclareMathSymbol{\dlsh}                  {3}{mathb}{"EA}
\DeclareMathSymbol{\drsh}                  {3}{mathb}{"EB}
\DeclareMathSymbol{\looparrowdownleft}     {3}{mathb}{"EE}
\DeclareMathSymbol{\looparrowdownright}    {3}{mathb}{"EF}
\DeclareMathSymbol{\curvearrowleftright}   {3}{mathb}{"F2}
\DeclareMathSymbol{\curvearrowbotleft}     {3}{mathb}{"F3}
\DeclareMathSymbol{\curvearrowbotright}    {3}{mathb}{"F4}
\DeclareMathSymbol{\curvearrowbotleftright}{3}{mathb}{"F5}
\DeclareMathSymbol{\leftsquigarrow}        {3}{mathb}{"F8}
\DeclareMathSymbol{\rightsquigarrow}       {3}{mathb}{"F9}
\DeclareMathSymbol{\leftrightsquigarrow}   {3}{mathb}{"FA}
\DeclareMathSymbol{\lefttorightarrow}      {3}{mathb}{"FC}
\DeclareMathSymbol{\righttoleftarrow}      {3}{mathb}{"FD}
\DeclareMathSymbol{\uptodownarrow}         {3}{mathb}{"FE}
\DeclareMathSymbol{\downtouparrow}         {3}{mathb}{"FF}
\DeclareMathSymbol{\varhash}       {0}{mathb}{"23}
\DeclareMathSymbol{\sqSubset}       {3}{mathb}{"94}
\DeclareMathSymbol{\sqSupset}       {3}{mathb}{"95}
\DeclareMathSymbol{\nsqSubset}      {3}{mathb}{"96}
\DeclareMathSymbol{\nsqSupset}      {3}{mathb}{"97}
\DeclareMathAccent{\widecheck}    {0}{mathx}{"71}

\DeclareMathOperator\colim{colim}
\DeclareMathOperator\llim{lim}
\let\lim\llim

\newcommand{\too}[1][]{\ensuremath{\overset{#1}{\longrightarrow}}}

\let\into\hookrightarrow

\mdef\we{\overset{\sim}{\longrightarrow}}
\mdef\leftwe{\overset{\sim}{\longleftarrow}}

\let\xto\xrightarrow

\def\rightarrowtailfill@{\arrowfill@{\Yright\joinrel\relbar}\relbar\rightarrow}
\newcommand\xrightarrowtail[2][]{\ext@arrow 0055{\rightarrowtailfill@}{#1}{#2}}

\def\twoheadrightarrowfill@{\arrowfill@{\relbar\joinrel\relbar}\relbar\twoheadrightarrow}
\newcommand\xtwoheadrightarrow[2][]{\ext@arrow 0055{\twoheadrightarrowfill@}{#1}{#2}}

\def\slashedarrowfill@#1#2#3#4#5{%
  $\m@th\thickmuskip0mu\medmuskip\thickmuskip\thinmuskip\thickmuskip
   \relax#5#1\mkern-7mu%
   \cleaders\hbox{$#5\mkern-2mu#2\mkern-2mu$}\hfill
   \mathclap{#3}\mathclap{#2}%
   \cleaders\hbox{$#5\mkern-2mu#2\mkern-2mu$}\hfill
   \mkern-7mu#4$%
}
\def\rightslashedarrowfill@{%
  \slashedarrowfill@\relbar\relbar\mapstochar\rightarrow}
\newcommand\xslashedrightarrow[2][]{%
  \ext@arrow 0055{\rightslashedarrowfill@}{#1}{#2}}
\mdef\hto{\xslashedrightarrow{}}
\mdef\htoo{\xslashedrightarrow{\quad}}

\def\jd#1{\@jd#1\ej}
\def\@jd#1|-#2\ej{\@@jd#1,,\;\vdash\;\left(#2\right)}
\def\@@jd#1,{\@ifmtarg{#1}{\let\next=\relax}{\left(#1\right)\let\next=\@@@jd}\next}
\def\@@@jd#1,{\@ifmtarg{#1}{\let\next=\relax}{,\,\left(#1\right)\let\next=\@@@jd}\next}
\def\jdm#1{\@jdm#1\ej}
\def\@jdm#1|-#2\ej{\@@jd#1,,\\\vdash\;\left(#2\right)}

\long\def\my@drawfill#1#2;{%
\@skipfalse
\fill[#1,draw=none] #2;
\@skiptrue
\draw[#1,fill=none] #2;
}
\newif\if@skip
\newcommand{\skipit}[1]{\if@skip\else#1\fi}
\newcommand{\drawfill}[1][]{\my@drawfill{#1}}

\newcounter{nodemaker}
\newcommand{\drpullback}[1][dr]{\ar[#1,phantom,near start,"\lrcorner"]}

\newif\ifhyperref
\@ifpackageloaded{hyperref}{\hyperreftrue}{\hyperreffalse}
\iftac
  \let\your@state\state
  \def\state#1{\my@state#1}
  \def\my@state#1.{\gdef\currthmtype{#1}\your@state{#1.}}
  \let\your@staterm\staterm
  \def\staterm#1{\my@staterm#1}
  \def\my@staterm#1.{\gdef\currthmtype{#1}\your@staterm{#1.}}
  \let\@defthm\newtheorem
  \def\switchtotheoremrm{\let\@defthm\newtheoremrm}
  \def\defthm#1#2#3{\@defthm{#1}{#2}} 
  \let\your@section\section
  \def\section{\gdef\currthmtype{section}\your@section}
  \def\currthmtype{}
  \ifhyperref
    \def\autoref#1{\ref*{label@name@#1}~\ref{#1}}
  \else
    \def\autoref#1{\ref{label@name@#1}~\ref{#1}}
  \fi
  \AtBeginDocument{%
    \let\old@label\label%
    \def\label#1{%
      {\let\your@currentlabel\@currentlabel%
        \edef\@currentlabel{\currthmtype}%
        \old@label{label@name@#1}}%
      \old@label{#1}}
  }
  \let\cref\autoref
\else\ifcref
  \def\defthm#1#2#3{%
    \newaliascnt{#1}{thm}
    \newtheorem{#1}[#1]{#2}
    \aliascntresetthe{#1}
    \crefname{#1}{#2}{#3}
  }
\else
  \ifhyperref
    \def\defthm#1#2#3{
      \newtheorem{#1}{#2}[section]%
      \expandafter\def\csname #1autorefname\endcsname{#2}%
      \expandafter\let\csname c@#1\endcsname\c@thm}
  \else
    \def\defthm#1#2#3{\newtheorem{#1}[thm]{#2}} 
    \ifx\SK@label\undefined\let\SK@label\label\fi
    \let\old@label\label
    \let\your@thm\@thm
    \def\@thm#1#2#3{\gdef\currthmtype{#3}\your@thm{#1}{#2}{#3}}
    \def\currthmtype{}
    \def\label#1{{\let\your@currentlabel\@currentlabel\def\@currentlabel%
        {\currthmtype~\your@currentlabel}%
        \SK@label{#1@}}\old@label{#1}}
    \def\autoref#1{\ref{#1@}}
  \fi
  \let\cref\autoref
\fi\fi

\newtheorem{thm}{Theorem}[section]
\ifcref
  \crefname{thm}{Theorem}{Theorems}
\else
  
\fi
\defthm{cor}{Corollary}{Corollaries}
\defthm{prop}{Proposition}{Propositions}
\defthm{lem}{Lemma}{Lemmas}
\defthm{sch}{Scholium}{Scholia}
\defthm{assume}{Assumption}{Assumptions}
\defthm{claim}{Claim}{Claims}
\defthm{conj}{Conjecture}{Conjectures}
\defthm{hyp}{Hypothesis}{Hypotheses}
\iftac\switchtotheoremrm\else\theoremstyle{definition}\fi
\defthm{defn}{Definition}{Definitions}
\defthm{notn}{Notation}{Notations}
\iftac\switchtotheoremrm\else\theoremstyle{remark}\fi
\defthm{rmk}{Remark}{Remarks}
\defthm{eg}{Example}{Examples}
\defthm{egs}{Examples}{Examples}
\defthm{ex}{Exercise}{Exercises}
\defthm{ceg}{Counterexample}{Counterexamples}

\ifcref
  \crefformat{section}{\S#2#1#3}
  \Crefformat{section}{Section~#2#1#3}
  \crefrangeformat{section}{\S\S#3#1#4--#5#2#6}
  \Crefrangeformat{section}{Sections~#3#1#4--#5#2#6}
  \crefmultiformat{section}{\S\S#2#1#3}{ and~#2#1#3}{, #2#1#3}{ and~#2#1#3}
  \Crefmultiformat{section}{Sections~#2#1#3}{ and~#2#1#3}{, #2#1#3}{ and~#2#1#3}
  \crefrangemultiformat{section}{\S\S#3#1#4--#5#2#6}{ and~#3#1#4--#5#2#6}{, #3#1#4--#5#2#6}{ and~#3#1#4--#5#2#6}
  \Crefrangemultiformat{section}{Sections~#3#1#4--#5#2#6}{ and~#3#1#4--#5#2#6}{, #3#1#4--#5#2#6}{ and~#3#1#4--#5#2#6}
  \crefformat{appendix}{Appendix~#2#1#3}
  \Crefformat{appendix}{Appendix~#2#1#3}
  \crefrangeformat{appendix}{Appendices~#3#1#4--#5#2#6}
  \Crefrangeformat{appendix}{Appendices~#3#1#4--#5#2#6}
  \crefmultiformat{appendix}{Appendices~#2#1#3}{ and~#2#1#3}{, #2#1#3}{ and~#2#1#3}
  \Crefmultiformat{appendix}{Appendices~#2#1#3}{ and~#2#1#3}{, #2#1#3}{ and~#2#1#3}
  \crefrangemultiformat{appendix}{Appendices~#3#1#4--#5#2#6}{ and~#3#1#4--#5#2#6}{, #3#1#4--#5#2#6}{ and~#3#1#4--#5#2#6}
  \Crefrangemultiformat{appendix}{Appendices~#3#1#4--#5#2#6}{ and~#3#1#4--#5#2#6}{, #3#1#4--#5#2#6}{ and~#3#1#4--#5#2#6}
  \crefformat{subappendix}{\S#2#1#3}
  \Crefformat{subappendix}{Section~#2#1#3}
  \crefrangeformat{subappendix}{\S\S#3#1#4--#5#2#6}
  \Crefrangeformat{subappendix}{Sections~#3#1#4--#5#2#6}
  \crefmultiformat{subappendix}{\S\S#2#1#3}{ and~#2#1#3}{, #2#1#3}{ and~#2#1#3}
  \Crefmultiformat{subappendix}{Sections~#2#1#3}{ and~#2#1#3}{, #2#1#3}{ and~#2#1#3}
  \crefrangemultiformat{subappendix}{\S\S#3#1#4--#5#2#6}{ and~#3#1#4--#5#2#6}{, #3#1#4--#5#2#6}{ and~#3#1#4--#5#2#6}
  \Crefrangemultiformat{subappendix}{Sections~#3#1#4--#5#2#6}{ and~#3#1#4--#5#2#6}{, #3#1#4--#5#2#6}{ and~#3#1#4--#5#2#6}
  \crefname{part}{Part}{Parts}
  \crefname{figure}{Figure}{Figures}
  \crefname{table}{Table}{Tables}
\fi

\iftac
  \let\qed\endproof
  \let\your@endproof\endproof
  \def\my@endproof{\your@endproof}
  \def\endproof{\my@endproof\gdef\my@endproof{\your@endproof}}
  \def\qedhere{\tag*{\endproofbox}\gdef\my@endproof{\relax}}
\fi

\iftac
  \def\pr@@f[#1]{\subsubsection*{\sc #1.}}
\fi

\def\thmqedhere{\expandafter\csname\csname @currenvir\endcsname @qed\endcsname}

\ifbeamer\else

\fi

\ifbeamer\else
  \setitemize[1]{leftmargin=2em}
  \setenumerate[1]{leftmargin=*}
\fi

\iftac
  \let\c@equation\c@subsection
\else
  \let\c@equation\c@thm
\fi
\numberwithin{equation}{section}

\ifcref\else
  \@ifpackageloaded{mathtools}{\mathtoolsset{showonlyrefs,showmanualtags}}{}
\fi

\alwaysmath{alpha}
\alwaysmath{beta}
\alwaysmath{gamma}
\alwaysmath{Gamma}
\alwaysmath{delta}
\alwaysmath{Delta}
\alwaysmath{epsilon}
\mdef\ep{\varepsilon}
\alwaysmath{zeta}
\alwaysmath{eta}
\alwaysmath{theta}
\alwaysmath{Theta}
\alwaysmath{iota}
\alwaysmath{kappa}
\alwaysmath{lambda}
\alwaysmath{Lambda}
\alwaysmath{mu}
\alwaysmath{nu}
\alwaysmath{xi}
\alwaysmath{pi}
\alwaysmath{rho}
\alwaysmath{sigma}
\alwaysmath{Sigma}
\alwaysmath{tau}
\alwaysmath{upsilon}
\alwaysmath{Upsilon}
\alwaysmath{phi}
\alwaysmath{Pi}
\alwaysmath{Phi}
\mdef\ph{\varphi}
\alwaysmath{chi}
\alwaysmath{psi}
\alwaysmath{Psi}
\alwaysmath{omega}
\alwaysmath{Omega}

\let\gm\gamma

\let\si\sigma

\alwaysmath{bot}
\alwaysmath{Bbbk}

\mdef\sihat{\widehat{\si}}
\mdef\tauhat{\widehat{\tau}}

\makeatother

\title{$\ast$-Autonomous Envelopes and Conservativity}
\author{Michael Shulman
  \institute{University of San Diego\\ San Diego, California, USA}
  \email{shulman@sandiego.edu}}
\def\titlerunning{$\ast$-autonomous envelopes}
\def\authorrunning{Michael Shulman}

\let\hom\multimap
\def\fsp#1{\mathord{\Asterisk}#1}
\let\coten\invamp
\let\unit\done
\def\d#1{#1^*}
\def\dd#1{#1^{**}}
\autodefs{\fMod\fEnv\cSh\cShEnv\fSet\dChu\cGl\dGl\fChu\fPoly\fGl}
\let\gl\fGl

\mdef\spoly{\ast\fPoly}
\mdef\supoly{\fPoly_{\le 1}}
\mdef\ssU{U^*_{\le 1}}
\def\mod#1{\fMod_{#1}}
\mdef\modp{\fMod_\cP}
\def\env#1{\fEnv_{#1}}
\mdef\envp{\fEnv_\cP}
\def\pj{(\cP,\cJ)}
\mdef\modpj{\fMod_{\pj}}
\mdef\envpj{\fEnv_{\pj}}
\mdef\pjk{(\cP,\cJ,\cK)}
\mdef\cjk{(\fsp\cC,\cJ,\cK)}
\mdef\modpjk{\fMod_{\pjk}}
\mdef\modcjk{\fMod_{\cjk}}
\mdef\envpjk{\fEnv_{\pjk}}
\mdef\envcjk{\fEnv_{\cjk}}
\let\chu\fChu
\def\tenj{\mathbin{\ten_\cJ}}

\def\noy#1{{}_{#1}{\reflectbox{\yon}}}
\def\scriptnoy#1{{}_{#1}{\reflectbox{$\scriptstyle\yon$}}}
\def\Umulti{U_{\mathrm{multi}}}

\def\p{^+}
\def\m{^-}

\def\o{^1}

\def\uA{\underline{A}}
\def\uf{\underline{\smash{f}}}
\mdef{\bbot}{\mathrlap{\bot}\,\bot}
\def\mamll{MA$\!^{\raisebox{1.5pt}{$\scriptscriptstyle-$}}$LL\xspace}
\def\imamll{IMA$\!^{\raisebox{1.5pt}{$\scriptscriptstyle-$}}$LL\xspace}
\def\Pl(#1|#2;#3){\cP\big(#1\mid #2\mathbin{;} #3\big)}
\def\El(#1|#2;#3){\cE\big(#1\mid #2\mathbin{;} #3\big)}
\def\Ul(#1|#2;#3){\cU\big(#1\mid #2\mathbin{;} #3\big)}
\def\linhom#1(#2|#3;#4){#1\big(#2\mid #3\mathbin{;} #4\big)}
\def\Pnl(#1;#2){\cP\big(#1 \mathbin{;} #2\big)}
\def\Enl(#1;#2){\cE\big(#1 \mathbin{;} #2\big)}
\def\nonlinhom#1(#2;#3){#1\big(#2 \mathbin{;} #3\big)}

\def\shift[#1|#2;#3]{[#1\,|\,#2\,;\,#3]}
\usepackage{bold-extra}         
\begin{document}
\maketitle

\begin{abstract}
  We prove 2-categorical conservativity for any $\{\mathbf{0},\top\}$-free fragment of MALL over its corresponding intuitionistic version: that is, that the universal map from a closed symmetric monoidal category to the $\ast$-autonomous category that it freely generates is fully faithful, and similarly for other doctrines.
  This implies that linear logics and graphical calculi for $\ast$-autonomous categories can also be interpreted canonically in closed symmetric monoidal categories.

  In particular, every closed symmetric monoidal category can be fully embedded in a $\ast$-autonomous category, preserving both tensor products and internal-homs.
  In fact, we prove this directly first with a Yoneda-style embedding (an enhanced ``Hyland envelope'' that can be regarded as a polycategorical form of Day convolution), and deduce 2-conservativity afterwards from Hyland--Schalk double gluing and a technique of Lafont.
  The same is true for other fragments of $\ast$-autonomous structure, such as linear distributivity, and the embedding can be enhanced to preserve any desired family of nonempty limits and colimits.
\end{abstract}



\section{Introduction}
\label{sec:introduction}

Of course, classical logic is not conservative over intuitionistic logic.
The linear situation is subtler: it was shown by~\cite{schellinx:syntac-ll} that classical multiplicative-additive linear logic (MALL) is not conservative over its intuitionistic variant (IMALL), but if either $\mathbf{0}$ or $\hom$ is removed then conservativity obtains.

For categorical models it is natural to ask for a stronger \emph{2-dimensional conservativity} (a.k.a.\ ``abstract full completeness''), i.e.\ is the universal functor from a model of IMALL to a model of MALL (or some fragments thereof) fully faithful?
This would imply that the more expressive theory MALL can consistently and unambiguously be used to reason about categorical models of the less expressive IMALL.
In particular, circuit diagrams, proof nets, and term calculi for $\ast$-autonomous categories (models of MLL), such as those of~\cite{bcst:natded-coh-wkdistrib,cs:pfth-bill,dp:proofnetcats,troelstra:lec-ll,reddy:acceptors,reddy:dirprolog}, could be used to reason about closed symmetric monoidal categories (models of IMLL).
This would be useful because the $\ast$-autonomous isomorphism $A\hom B \cong \d{(A \ten \d B)}$\footnote{Note that the simpler $A\hom B \cong \d{A} \ten B$ holds only in a compact closed category, and a symmetric monoidal category cannot be fully embedded in a compact closed one unless it is \emph{traced}~\cite{jsv:traced-moncat}.} enables a simple graphical representation of internal-homs using bent strings, in contrast to the additional ``clasps'' and ``bubbles'' (as in~\cite{bs:rosetta}) that may appear na\"{i}vely to be needed.

Note that to have \emph{some} way to interpret $\ast$-autonomous graphical calculi in closed symmetric monoidal categories, it would suffice to show that any category of the latter sort embeds fully-faithfully in \emph{some} category of the former sort.
But in this case it could happen, in principle, that the interpretation depends on the embedding chosen.
Our 2-conservativity remedies this: because the \emph{universal} map from a closed symmetric monoidal category to a $\ast$-autonomous category is fully faithful, the interpretation obtained from this universal embedding will necessarily coincide with that obtained from any other embedding.

Since our proof of 2-conservativity is constructive, it is also possible to see it as a sort of ``normalization'' theorem: given any syntactic expression for an MLL-morphism between IMLL-types, we can normalize it to an expression for an IMLL-morphism.

Lafont~\cite{lafont:thesis} gave a general method for proving 2-conservativity between intuitionistic doctrines, using Artin gluing (the semantic form of logical relations) and the Yoneda embedding.
When working with ``classical'' theories such as MLL, gluing must be replaced by the \emph{double gluing} construction of~\cite{hs:glue-orth-ll,hasegawa:glueing-cll}, but (contrary to claims in \textit{op.~cit.}) the ordinary Yoneda embedding is also insufficient in this case.
(Double gluing involves logical relations on both objects and their duals, but the Yoneda embedding carries no information about duals; thus the relevant functor to the gluing construction fails to preserve internal-homs. See \cref{rmk:mistake}.)

An appropriate modified Yoneda lemma uses the \emph{envelope} of~\cite{hyland:pfthy-abs}, which combines a sort of ``presheaf'' on a polycategory with a Chu construction.
It can also be thought of as an enhancement of Isbell duality~\cite{isbell:soc}, which uses the hom-functor $\cC(-,-)$ as a \emph{canonical} dualizing object $\bot$ between presheaf and copresheaf categories for which all representables are reflexive (i.e.\ $A \cong (A\hom \bot)\hom \bot$), to a Chu situation where the dualizing object lives in the same category as the objects being dualized.

We modify Hyland's envelope so that the embedding preserves any desired tensor and cotensor products and nonempty limits and colimits, by adapting the standard trick~\cite{fk:cts-frs,kelly:enriched} for making Yoneda embeddings preserve colimits.
(The limits and colimits must be nonempty, because only nonempty limits and colimits in closed symmetric monoidal categories are polycategorical.)
This yields a Yoneda-type embedding for polycategories, including closed symmetric monoidal categories.
Combining it with double gluing and Lafont's method, we obtain 2-conservativity for any fragments of IMALL and MALL lacking $\mathbf{0}$ and $\top$, and other pairs of theories such as MALL and its negation-free fragment (corresponding to linearly distributive categories; see~\cite{bcst:natded-coh-wkdistrib}).
Thus, $\ast$-autonomous calculi can be used for closed symmetric monoidal categories, giving a new perspective on why the same ``Kelly-MacLane graphs'' (and enhanced versions incorporating units) appear in coherence for closed symmetric monoidal categories and for $\ast$-autonomous categories~\cite{km:coh-closed,trimble:thesis,blute:proof-nets,blute:ll-coh,hughes:free-staraut}.

In particular, we obtain a semantic proof of the conservativity of multiplicative linear logic over its intuitionistic variant, similar to the results of of~\cite{schellinx:syntac-ll}.
Our methods are both more and less powerful than the syntactic ones of~\cite{schellinx:syntac-ll}.
On one hand, in addition to yielding a 2-categorical statement, our notion of full-faithfulness is \emph{polycategorical} rather than \emph{multicategorical}: syntactically this means that if \emph{any} MLL sequent $\Gamma\vdash\Delta$ between IMLL types is derivable in MLL, then it is also derivable in IMLL --- and \emph{therefore} $\Delta$ must consist of only one type.
By contrast, the conservativity of~\cite{schellinx:syntac-ll} is only multicategorical: the fact that $\Delta$ contains only one type must be assumed at the outset.

On the other hand, sometimes multicategorical conservativity can hold while polycategorical conservativity fails.
In particular, combining our results with those of~\cite{schellinx:syntac-ll}, we see that this is the case when a terminal object $\top$ (but not an initial object $\mathbf{0}$) is included in MLL and IMLL.
We do not know whether our semantic methods can be adapted to such cases.


\section{Adding Duals to Polycategories}
\label{sec:starpoly}

To use Lafont's technique, we require a fully faithful Yoneda-type embedding of a closed symmetric monoidal category into an $\ast$-autonomous category.
We will construct this by using polycategories.

A symmetric\footnote{All our multi- and polycategories will be symmetric, so we henceforth drop the adjective.} \textbf{polycategory}~\cite{szabo:polycats} semantically represents the judgmental structure of classical linear logic, with hom-sets $\cP(\Gamma;\Delta)$ where $\Gamma$ and $\Delta$ are lists of objects, and compositions such as $\cP(\Gamma,A;\Delta) \times \cP(\Sigma;\Pi,A) \to \cP(\Gamma,\Sigma;\Delta,\Pi)$.
A \textbf{multicategory} can be defined as a polycategory that is \textbf{co-unary}, i.e.\ all codomains have one object.
In particular, a closed symmetric monoidal category $(\cC,\ten,\unit,\hom)$ can be regarded as a multicategory that is \textbf{representable and closed}, meaning there are objects with universal properties:
\begin{gather*}
  \cC(\Gamma;C) \cong \cC(\Gamma,\unit;C)\qquad
  \cC(\Gamma,A,B;C) \cong \cC(\Gamma,A\ten B; C)\qquad
  \cC(\Gamma,A;B) \cong \cC(\Gamma; A\hom B).
\end{gather*}
Similarly, a $\ast$-autonomous category \cE can be regarded as a representable\footnote{Note that a representable multicategory is not representable as a co-unary polycategory.} polycategory with duals, i.e.\ having objects with universal properties:
\begin{gather*}
  \cE(\Gamma,A,B;\Delta) \cong \cE(\Gamma,A\ten B; \Delta) \qquad
  \cE(\Gamma;\Delta,A,B) \cong \cE(\Gamma; \Delta,A\coten B) \\
  \cE(\Gamma;\Delta) \cong \cE(\Gamma,\unit;\Delta) \qquad
  \cE(\Gamma;\Delta) \cong \cE(\Gamma;\Delta,\bot)\qquad
  \cE(\Gamma,A; \Delta) \cong \cE(\Gamma;\Delta,\d A)
\end{gather*}
A \textbf{$\ast$-polycategory}~\cite{hyland:pfthy-abs} is a polycategory with specified strictly involutive duals ($\dd{A} = A$).
In particular, a representable $\ast$-polycategory is $\ast$-autonomous.

The forgetful functor from $\ast$-polycategories to polycategories has a left adjoint $\cP\mapsto\fsp\cP$.
The objects of $\fsp\cP$ consist of two copies of the objects of \cP, denoted $A$ and $\Abar$ respectively.
The morphisms are determined by saying that
\[ \fsp\cP(\Gamma,\overline\Pi; \Delta, \overline\Sigma) = \cP(\Gamma,\Sigma; \Delta,\Pi). \]
Composition is inherited from \cP, perhaps in the other order.
For instance, if $f\in \fsp\cP(A,\Bbar;C,\Dbar) = \cP(A,D;C,B)$ and $g\in \fsp\cP(C,\Xbar;\Bbar,Y) = \cP(C,B;\,Y,X)$ then $g\circ^{\fsp\cP}_C f = g\circ_C^\cP f$ and $f \circ_{\d B}^{\fsp\cP} g = g\circ_B^\cP f$.
Thus $\cP$ embeds fully-faithfully in $\fsp\cP$.

Polycategory functors preserve duals, so if \cP has duals the map $\cP \to \fsp\cP$ is essentially surjective, hence an equivalence.
Thus any $\ast$-autonomous category is equivalent to a representable $\ast$-polycategory.\footnote{This was shown by~\cite{chs:coh-staut} using a \emph{right} adjoint instead of our left adjoint.  In fact \mbox{$\ast$-polycategories} are both 2-monadic and 2-comonadic over the 2-category of polycategories, functors, and natural isomorphisms, and the 2-monad and 2-comonad are pseudo-idempotent~\cite{kl:property-like}.}

Any tensor product $A\ten B$ in \cP is also one in $\fsp\cP$, while $\overline{(A\ten B)}$ is a cotensor product $\Abar \coten \Bbar$ in $\fsp\cP$.
The situation for cotensor products $A\coten B$ is dual, while that for units and counits is similar.
However, even if \cP has all tensor and/or cotensor products, $\fsp\cP$ will not in general have $\d A \ten B$ or $\d A \coten B$.

If \cC is a multicategory regarded as a co-unary polycategory, then $\fsp\cC(\Gamma,\d\Pi;\Delta,\d\Sigma)$ is nonempty just when $|\Delta\cup\Pi|=1$.
(This left adjoint to the forgetful functor from $\ast$-polycategories to multicategories appears in~\cite{dh:dk-cyc-opd}.)
If \cC is also closed, then $A\hom B$ is a cotensor product $\d A \coten B$ in $\fsp\cC$, for:
\begin{align*}
  \fsp\cC(\Gamma,\overline\Pi; \Delta, \overline\Sigma, A\hom B)
  &\cong \cC(\Gamma,\Sigma; \Pi,\Delta, A\hom B)\\
  \fsp\cC(\Gamma,\overline\Pi; \Delta, \overline\Sigma, \Abar, B)
  &\cong \cC(\Gamma,\Sigma, A ; \Pi,\Delta, B)
\end{align*}
and both right-hand sides are nonempty only if $\Pi = \Delta=\emptyset$, in which case they are naturally isomorphic by the universal property of $A\hom B$ in \cC.
Let $\Umulti$ denote the forgetful functor from polycategories to multicategories; then we have shown:

\begin{thm}\label{thm:smcc-poly}
  If \cC is closed symmetric monoidal, there is a $\ast$-polycategory $\fsp\cC$ and a fully faithful functor $\cC \to \Umulti(\fsp\cC)$ that preserves tensor products (including the unit) and takes internal-homs $A\hom B$ to cotensor products $\Abar \coten B$.\qed
\end{thm}

Therefore, to embed \cC in a $\ast$-autonomous category preserving both tensor products and internal-homs, it will suffice to embed the $\ast$-polycategory $\fsp\cC$ in a $\ast$-autonomous category preserving \emph{those tensor and cotensor products that exist}.

\section{Modules}
\label{sec:saut}

Let \cP be a polycategory.
Following~\cite[\S5]{hyland:pfthy-abs}, a \textbf{\cP-module} is a family of sets $\cU(\Gamma;\Delta)$ with symmetric group actions and left and right actions by \cP
\begin{equation*}
  \cP(\Pi;\Sigma,A) \times \cU(A,\Gamma;\Delta) \to \cU(\Pi,\Gamma;\Sigma,\Delta)
  \qquad
   \cU(\Gamma;\Delta,A)\times \cP(A,\Pi;\Sigma) \to \cU(\Gamma,\Pi;\Delta,\Sigma).
\end{equation*}
satisfying the same associativity and unit laws as the composition in a polycategory.

\begin{egs}\ 
  \begin{enumerate}
  \item The hom-sets $\cP(\Gamma;\Delta)$ form a \textbf{tautological} \cP-module, denoted \cP.
  \item A \textbf{shifted} module $\cU[\Pi;\Sigma]$ is defined by $\cU[\Pi;\Sigma](\Gamma;\Delta) = \cU(\Gamma,\Pi;\Delta,\Sigma)$.
  \item For $A\in \cP$, we have the \textbf{representable} modules\footnote{The symbol $\yon$ is the hiragana kana for ``yo''; its use for Yoneda embeddings was introduced in~\cite{jfs:twqft}.}
    \(\yon_A = \cP[\,;A]\)  and \(\noy{A} = \cP[A;\,] \).
  \end{enumerate}
\end{egs}

\begin{thm}[{\cite[\S5.2]{hyland:pfthy-abs}}]\label{thm:mod-smcc}
  The category $\modp$ of \cP-modules is closed symmetric monoidal and complete and cocomplete.
\end{thm}
\begin{proof}
  Let $F_{\ten} \Umulti \fsp \cP$ be the free symmetric strict monoidal category on the underlying multicategory of $\fsp\cP$.
  Its objects are finite lists of objects and \cP and their formal duals, but by symmetry each is isomorphic to one of the form $(\Gamma,\overline\Delta)$ where $\Gamma$ and $\Delta$ consist of objects of \cP.
  A \cP-module \cU is then equivalent to an ordinary presheaf on $F_{\ten} \Umulti \fsp \cP$ defined by $(\Gamma,\overline\Delta)\mapsto \cU(\Gamma;\Delta)$.
  But now since $\modp$ is a presheaf category on a symmetric monoidal domain, it is complete and cocomplete and inherits a closed symmetric monoidal Day convolution~\cite{day:closed} monoidal structure.
\end{proof}

We will often consider $\modp$ as a multicategory.
In this case, a module morphism $(\cU_1,\cdots, \cU_n) \to \cV$ consists of functions
\[ \cU_1(\Gamma_1;\Delta_1) \times \cdots \times \cU_n(\Gamma_n;\Delta_n) \to \cV(\Gamma_1,\dots,\Gamma_n; \Delta_1,\dots,\Delta_n) \]
that commute with the symmetric group actions and the actions of \cP.
The unit module \cI is defined by $\cI(\,;\,) = 1$ and all other sets empty, so a nullary morphism $()\to \cV$ is just an element of $\cV(\,;\,)$. 
And in the internal-hom of modules, $(\cU \hom \cV)(\Gamma;\Delta)$ is the set of module morphisms from $\cU$ to $\cV[\Gamma;\Delta]$.

We regard $\modp$ as a polycategorical ``presheaf category'', justified by Yoneda lemmas:

\begin{thm}[Polycategorical Yoneda lemmas]\label{thm:yoneda}
  We have natural isomorphisms
  \begin{alignat}{2}
    \modp(\yon_A; \cV) &\cong \cV(A;\,) &\qquad
    \modp(\noy{A};\cV) &\cong \cV(\,;A)\label{eq:yoneda1} \\
    \modp(\Gamma,\yon_A ; \cV) &\cong \modp(\Gamma; \cV[A;\,]) &\qquad
    \modp(\Gamma,\noy{A} ; \cV) &\cong \modp(\Gamma; \cV[\,;A])\label{eq:yoneda4}\\
    (\yon_A \hom \cV) &\cong \cV[A;\,] &\qquad
    (\noy{A} \hom \cV) &\cong \cV[\,;A] \label{eq:yoneda3}
  \end{alignat}
\end{thm}
\begin{proof}
  This follows formally from properties of Day convolution, but we can also give an explicit proof.
  Since $1_A \in \cP(A;A) = \yon_A(A;\,)$, any $\phi:\yon_A \to \cV$ induces $\phi(1_A) \in \cV(A;\,)$.
  Conversely, from $x\in \cV(A;\,)$ we define $\psi_x:\yon_A \to \cV$ by:
  \[ \Big(f\in \yon_A(\Gamma;\Delta) = \cP(\Gamma;\Delta,A)\Big)
    \mapsto
    \Big(x \circ_A f \in \cV(\Gamma;\Delta) \Big).
  \]
  The associativity of the \cP-action on \cV ensures that this is a \cP-module morphism.
  Clearly $\psi_x(1_A) = x \circ_A 1_A = x$, while on the other side we have
  \begin{equation*}
    \psi_{\phi(1_A)}(f) = \phi(1_A) \circ_A f
    = \phi (1_A \circ_A f)
    = \phi(f).
  \end{equation*}
  This, and a dual calculation, proves~\eqref{eq:yoneda1}.
  For~\eqref{eq:yoneda3}, we have
  \begin{equation*}
    (\yon_A \hom \cV)(\Gamma;\Delta)
    = \modp(\yon_A, \cV[\Gamma;\Delta])\\
    \cong
    \cV[\Gamma;\Delta](A;\,)  
    =
    \cV(A,\Gamma;\Delta)
    \cong \cV[A;\,](\Gamma;\Delta)
  \end{equation*}
  and dually.
  Finally,~\eqref{eq:yoneda4} follows from~\eqref{eq:yoneda3} and the universal property of $\hom$.
\end{proof}

\begin{cor}\label{thm:dual-rep}
  $(\yon_A \hom \cP) \cong \noy{A}$ and $(\noy{A}\hom \cP) \cong \yon_A$.\qed
\end{cor}

In particular, both kinds of representable module are ``reflexive'': $\yon_A \cong ((\yon_A \hom \cP)\hom\cP)$ and $\noy A \cong ((\noy A \hom \cP)\hom\cP)$.

\begin{cor}\label{thm:yoneda-emb}
  There are full embeddings of multicategories
  \[\yon : \Umulti(\cP)\to\modp \qquad\text{and}\qquad \noy{} : \Umulti(\cP\op) \to \modp.\]
\end{cor}
\begin{proof}
  From \cref{thm:yoneda}, we have
  \begin{align*}
    \modp(\yon_{A_1},\dots,\yon_{A_n};\yon_B) &\cong \yon_B(A_1,\dots,A_n;\,) = \cP(A_1,\dots,A_n;B)\\
    \modp(\noy{A_1},\dots,\noy{A_n};\noy B) &\cong \noy B(\,;A_1,\dots,A_n) = \cP(B;A_1,\dots,A_n).
  \end{align*}
  Functoriality is easy to check.
\end{proof}

\section{The Polycategorical Chu Construction}
\label{sec:poly-chu}

\noindent
Let \cE be a multicategory; a \textbf{presheaf} $\Bbbk$ on \cE can be defined equivalently as either:
\begin{enumerate}
\item A module (as in \cref{sec:saut}) over \cE \emph{qua} co-unary polycategory, whose only nonempty values are $\cU(\Gamma;\,)$.\label{item:pshf1}
\item An ordinary presheaf on the free symmetric strict monoidal category $F_{\ten}\cE$ generated by \cE.
\item An extension of \cE to a \emph{co-subunary polycategory} (i.e.\ all morphisms have codomain arity 0 or 1).\label{item:psh2}
\item A structured family of sets as in~\cite[\S2]{shulman:dialectica}; here we consider only set-valued presheaves.
\end{enumerate}
If $\Bbbk$ is a module, we sometimes write $\phi\in\Bbbk(\Gamma)$ as $\phi:\Gamma\to\Bbbk$.
This is not very abusive, since by \cref{thm:yoneda-emb} the set $\Bbbk(\Gamma)$ is isomorphic to $\mod{\cE}(\yon_\Gamma;\Bbbk)$, where $\yon_{A_1,\dots,A_n} = (\yon_{A_1},\dots,\yon_{A_n})$.


The following multicategorical Chu construction first appeared, to my knowledge, in~\cite{shulman:dialectica}, although~\cite[Example 1.8(2)]{cks:polybicats} contains a similar construction for bicategories.
It explains the Chu tensor product~\cite{chu:constr-app} by a universal property.

\begin{defn}
  The \textbf{Chu construction} $\chu(\cE,\Bbbk)$ 
  is the following $\ast$-polycategory:
  \begin{itemize}
  \item Its objects are triples $A=(A\p,A\m,\uA)$ where $A\p$ and $A\m$ are objects of \cE, and $\uA \in \cE(A\p,A\m;\Bbbk)$.
    We have $\d{(A\p,A\m,\uA)} = (A\m,A\p,\uA\sigma)$, where $\sigma$ denotes the permutation action.
  \item Its morphisms $f:(A_1,\dots,A_m) \to (B_1,\dots,B_n)$ are families of morphisms in \cE:
    \begin{alignat*}{2}
      f\p_j &: (A_1\p, \dots, A_m\p, B_1\m, \dots, \widehat{B_j\m},\dots B_n\m) \longrightarrow B_j\p
      &&\qquad (1 \le j \le n)\\
      f\m_i &: (A_1\p , \dots, \widehat{A_i\p},\dots A_m\p , B_1\m, \dots, B_n\m) \longrightarrow A_i\m
      &&\qquad (1\le i \le m)\\
      \uf &: (A_1\p,\dots,A_m\p,B_1\m,\dots,B_n\m) \longrightarrow \Bbbk
    \end{alignat*}
    (where hats indicate omitted entries) such that $\underline{\smash{B_j}}\circ_{B_j\p} f\p_j = \uf$ and $\underline{\smash{A_i}}\circ_{A_i\m} f\m_i = \uf$ (modulo permutations).
    If $m=n=0$, the only datum is $\uf : () \to \Bbbk$.
  \item The identity of $(A\p,A\m,\uA)$ is $(1_{A\p},1_{A\m},\uA)$; composition is induced from \cE.
  \end{itemize}
\end{defn}

\begin{thm}\label{thm:chu-rep}
  If \cE is representable and closed with pullbacks, and $\Bbbk=\yon_\ik$ is a representable presheaf, then $\chu(\cE,\Bbbk)$ is a representable polycategory, hence a $\ast$-autonomous category, and coincides with the classical Chu construction~\cite{chu:construction,chu:constr-app}.
\end{thm}
\begin{proof}
  The usual formulas $\unit = (\unit,\ik,\ell)$ and $A\ten B = (A\p\ten B\p, P, \rho)$ can be verified to have the correct universal properties, where $P$ is the pullback
  \[
    \begin{tikzcd}[baseline=-7mm]
      P \ar[r] \ar[d] \drpullback & A\p \hom B\m \ar[d] \\
      B\p \hom A\m \ar[r] & (A\p\ten B\p)\hom \ik.
    \end{tikzcd}\qedhere
  \]
\end{proof}

In addition, the multicategorical Chu construction itself has a simple universal property, which generalizes and strictifies that of~\cite{pavlovic:chu-i}.
Let \spoly denote the category of $\ast$-polycategories, and \supoly that of co-subunary polycategories. 

\begin{thm}\label{thm:chu-adjt}
  $\chu$ is right adjoint to the forgetful $\ssU: \spoly \to \supoly$, and the adjunction is comonadic.
\end{thm}
\begin{proof}
  The counit $\ssU\chu(\cE) \to \cE$ extracts $A\p$ from $A$, $f_1\p$ from a morphism $f:(A_1,\dots,A_m) \to B_1$, and $\uf$ from a morphism $f:(A_1,\dots,A_n) \to ()$.
  The unit $\cP \to \chu(\ssU(\cP))$ sends an object $A$ to $(A,\d A, \ep_A)$, where $\ep_A \in \cP(A,\d A;\,)$ is the duality counit, and a morphism $f\in \cP(\Gamma;\Delta)$ to the family of all its co-subunary duality images.
  The coalgebras for the induced comonad are \emph{co-subunary $\ast$-polycategories}, which by~\cite[\S7]{shulman:dialectica} are equivalent to ordinary $\ast$-polycategories.
\end{proof}


\section{Envelopes}
\label{sec:envelopes}

Let \cP be a polycategory; we now describe Hyland's Yoneda-type embedding of \cP.

\begin{defn}[{\cite[\S5]{hyland:pfthy-abs}}]
  The \textbf{envelope} of \cP is the Chu construction
  \[ \envp=\chu(\modp,\cP).\]
\end{defn}

Thus, an object of $\envp$ is two modules \cU, \cV with a module map $(\cU,\cV) \to \cP$.
By \cref{thm:chu-rep}, $\envp$ is \mbox{$\ast$-autonomous}, and contains $\modp$ as a symmetric monoidal full subcategory via $\cU \mapsto (\cU, \cU\hom\cP, \mathsf{ev})$.
Hence it also contains $\Umulti(\cP)$ as a full sub-multicategory via $A \mapsto (\yon_A, \noy{A}, \gamma_A)$, where $\gamma_A : (\yon_A, \noy{A}) \to \cP$ is composition in \cP.
In fact, Hyland showed:

\begin{thm}[{\cite[\S5]{hyland:pfthy-abs}}]\label{thm:yoneda-poly}
  The assignment $A\mapsto (\yon_A, \noy{A}, \gamma_A)$ extends to a full embedding of \emph{polycategories} $\cP \into \envp$.
\end{thm}
\begin{proof}
  A morphism $g:(\gamma_{A_1},\dots,\gamma_{A_m}) \to (\gamma_{B_1},\dots,\gamma_{B_n})$ in $\envp$ is a compatible family:
  \begin{alignat*}{2}
    g\p_j :&& (\yon_{A_1},\dots,\yon_{A_m},\noy{B_1},\dots,\widehat{\noy{B_j}},\dots,\noy{B_n}) &\to \yon_{B_j}\\
    g\m_i :&& (\yon_{A_1},\dots,\widehat{\yon_{A_i}},\dots,\yon_{A_m},\noy{B_1},\dots,\noy{B_n}) &\to \noy{A_i}\\
    \ug :&& (\yon_{A_1},\dots,\yon_{A_m},\noy{B_1},\dots,\noy{B_n}) &\to \cP.
  \end{alignat*}
  By \eqref{eq:yoneda4}, each of these is equivalent to a morphism $(A_1,\dots,A_m) \to (B_1,\dots, B_n)$ in \cP, and the compatibility conditions say they all correspond to the same such morphism.
  Functoriality is straightforward.
\end{proof}

Thus, any polycategory \cP embeds fully-faithfully in a $\ast$-autonomous category, and indeed in a Chu construction.
An explicit way to extract a morphism in \cP from a morphism $g:(\gamma_{A_1},\dots,\gamma_{A_m}) \to (\gamma_{B_1},\dots,\gamma_{B_n})$ is to evaluate the component
\[ \small \begin{tikzcd} \yon_{A_1}(A_1;) \times \cdots \times \yon_{A_m}(A_m;) \times \noy{B_1}(;B_1)\times\cdots\times \noy{B_n}(;B_n) \to \cP(A_1,\dots,A_m;B_1,\dots,B_n) 
  \end{tikzcd}
\]
of \ug at the identities $(1_{A_1},\dots,1_{A_m},1_{B_1},\dots,1_{B_n})$.

\begin{rmk}\label{rmk:isbell}
  If \cP is an ordinary category regarded as a unary co-unary polycategory, then the categories $[\cP\op,\fSet]$ and $[\cP,\fSet]$ are equivalent to the categories of modules whose only nonempty values have the form $\cU(A;\,)$ and $\cU(\,;A)$, respectively.
  The functor $(-)\hom \cP$ restricts to Isbell conjugation interchanging these two subcategories, and the Hyland envelope contains the Isbell envelope~\cite{isbell:soc}.
\end{rmk}

\section{Preserving Tensors and Cotensors}
\label{sec:sheaves}

Hyland's envelope does not preserve any tensor or cotensor products that exist in \cP,
but we can modify it to do so along the lines of~\cite{fk:cts-frs} and~\cite[\S3.12]{kelly:enriched}.

Suppose \cP is equipped with a set \cJ of tensor and cotensor products that exist (potentially including the nullary cases of a unit and/or counit), which we call \textbf{distinguished}.
Our intended example is $\cP=\fsp\cC$ from \cref{sec:starpoly}, with the tensor products and unit coming from \cC, and the cotensor products $(\Abar \coten B) = (A\hom B)$.

\begin{defn}
  A $\pj$-\textbf{module} is a module that respects the distinguished tensor and cotensor products.
  E.g.\ if $(A\ten B) \in \cJ$, the induced maps such as
  \( \cU(\Gamma,A\ten B;\Delta) \to \cU(\Gamma,A,B;\Delta)\)
   are isomorphisms.
\end{defn}

Let $\modpj\subseteq \modp$ consist of the $\pj$-modules.
Of course, \cP is a $\pj$-module, as is any shift of a $\pj$-module; thus $\yon_A,\noy{A}\in \modpj$.

\begin{thm}\label{thm:sh-emb}
  The embedding $\yon : \Umulti(\cP)\to\modpj$ preserves distinguished tensor products.
  Dually, the embedding $ \noy{} : \Umulti(\cP\op) \to \modpj$ takes distinguished cotensor products (which are tensor products in $\cP\op$) to tensor products.
\end{thm}
\begin{proof}
  Let \cU be a $\pj$-module and $\Gamma$ a list of $\pj$-modules. 
  Then by~\eqref{eq:yoneda4} and the assumption on \cU we have natural bijections:
  \begin{equation*}
    \modp(\Gamma,\yon_A,\yon_B\mathbin;\cU)
    \cong \modp(\Gamma\mathbin; \cU[A,B;\,])
    \cong \modp(\Gamma\mathbin ; \cU[A\ten B;\,])
    \cong \modp(\Gamma,\yon_{A\ten B}\mathbin; \cU).
  \end{equation*}
  Thus, $\yon_{A\ten B}$ is a tensor product $\yon_A \ten \yon_B$. 
  The dual statement is similar.
\end{proof}


\begin{thm}\label{thm:mod-ei}
  \begin{enumerate}
  \item $\modpj$ is a reflective subcategory of $\modp$.\label{item:ei1}
  \item If $\cU\in\modp$ and $\cV\in \modpj$ then $(\cU\hom\cV)\in\modpj$.\label{item:ei2}
  \item $\modpj$ has a closed symmetric monoidal structure such that the reflector is strong monoidal and the inclusion preserves internal-homs.\label{item:ei3}
  \end{enumerate}
\end{thm}
\begin{proof}
  For a list of objects $\Gamma$, write $\yon_\Gamma$ for the tensor product of all their representable modules $\yon_A$, and similarly $\noy{\Gamma}$ for the tensor product (not a cotensor product!\  $\modp$ has no cotensors) of their dual representables $\noy{A}$.
  In particular, $\yon_\Gamma \ten \noy{\Delta}$ is the ordinary representable presheaf at the object $(\Gamma,\overline\Delta) \in F_{\ten} \Umulti \fsp \cP$, so by the ordinary Yoneda lemma, a morphism $\yon_\Gamma \ten \noy{\Delta} \to \cU$ is the same as an element of $\cU(\Gamma;\Delta)$.
  It follows that \cU respects a tensor product $A\ten B$ if and only if every morphism $\yon_{\Gamma,A,B} \ten \noy{\Delta} \to \cU$ extends uniquely to a morphism $\yon_{\Gamma,A\ten B} \ten \noy{\Delta} \to \cU$:
  \[
    \begin{tikzcd}
      \yon_{\Gamma,A,B} \ten \noy{\Delta} \ar[r] \ar[d] & \cU\\
      \yon_{\Gamma,A\ten B} \ten \noy{\Delta} \ar[ur,dotted,"\exists!"']
    \end{tikzcd}
  \]
  and similarly for cotensor products.
  Thus, $\modpj$ is a \textbf{small-orthogonality class} in $\modp$, so its reflectivity follows from a standard ``small object argument'' iterative construction as in~\cite[1.36--1.38]{ar:loc-pres}.
  (Since $\modp$ is locally finitely presentable and the objects $\yon_{\Gamma,A,B} \ten \noy{\Delta}$ and $\yon_{\Gamma,A\ten B} \ten \noy{\Delta}$ are finitely presentable, this construction requires only countably many steps and is fully constructive.)

  This proves~\ref{item:ei1}.
  Now by definition
  \begin{equation*}
    (\cU \hom \cV)(\Gamma;\Delta) = \modp(\cU, \cV[\Gamma;\Delta]),
  \end{equation*}
  so~\ref{item:ei2} follows since $\cV$ is a $\pj$-module.
  Finally,~\ref{item:ei3} follows formally~\cite{day:refl-closed}; the tensor product $\tenj$ of $\modpj$ is the reflection of that of $\modp$.
\end{proof}

\begin{eg}
  By the formula for $\cU\ten \cV$ on p28 of~\cite{hyland:pfthy-abs}, any elements $u\in \cU(A;\,)$ and $v\in \cV(B;\,)$ induce an element of $(\cU\ten\cV)(A,B;\,)$, hence of $(\cU\tenj\cV)(A\ten B; \,)$ and thence $(\cU\tenj\cV)(C; \,)$ for any $C\to A\ten B$ in \cP.
  Thus we have a map
  \[ \cP(C, A\ten B) \times \cU(A;\,) \times \cV(B;\,) \too (\cU\tenj \cV)(C;\,) \]
  showing that $\tenj$ is similar to Day convolution~\cite{day:closed}.
\end{eg}

\begin{defn}
  Given $\pj$, its \textbf{envelope} is the Chu construction
  \[ \envpj=\chu(\modpj,\cP).\]
\end{defn}


\begin{thm}\label{thm:shenv}
  $\envpj$ is $\ast$-autonomous, contains \cP as a full sub-polycategory, and the inclusion preserves the distinguished tensor and cotensor products.
\end{thm}
\begin{proof}
  Since $\modpj$ is a full sub-multicategory of $\modp$, $\envpj$ is a full sub-polycategory of $\envp$.
  Since it contains the image of \cP, which is a full sub-polycategory of $\envp$ by \cref{thm:yoneda-poly}, the inclusion $\cP \into \envpj$ is also full.

  Now the embedding of $\modpj$ in $\envpj$, like that of any monoidal category in its Chu construction, preserves tensor products.
  Since the embedding of $\Umulti(\cP)$ in $\modpj$ preserves distinguished tensor products by \cref{thm:sh-emb}, so does the composite $\Umulti(\cP) \into \modpj\into \envpj$.
  Dually, the composite embedding 
  $\Umulti(\cP) \into \modpj\op \into \envpj$ preserves distinguished cotensor products, and by \cref{thm:dual-rep} the two embeddings coincide.
\end{proof}

Combining \cref{thm:smcc-poly,thm:shenv}, we obtain our first new embedding theorem.

\begin{thm}\label{thm:main}
  Any closed symmetric monoidal category \cC can be fully embedded in a $\ast$-autonomous category by a strong symmetric monoidal closed functor.
\end{thm}
\begin{proof}
  By \cref{thm:smcc-poly}, \cC embeds fully in $\fsp\cC$ preserving tensor products and taking $A\hom B$ to $\Abar \coten B$.
  Let \cJ consist of these tensor and cotensor products, and embed $\fsp\cC$ in the $\ast$-autonomous category $\env{(\fsp\cC,\cJ)}$, which by \cref{thm:shenv} preserves these tensor and cotensor products.
\end{proof}

\begin{rmk}
  The use of $\fsp\cC$ is not strictly necessary: we have $\mod{\fsp\cC} \simeq \mod{\cC}$, and the preservation of cotensors $\Abar \coten B$ can be expressed directly in terms of a $\cC$-module as a preservation of internal-homs, $\cU(\Gamma;\Delta,A\hom B) \cong \cU(\Gamma,A;\Delta,B)$.
  However, it is convenient to prove \cref{thm:shenv} by separate reduction to \cref{thm:sh-emb} and its dual, which requires expressing internal-homs as cotensor products.
\end{rmk}

\section{Preserving Limits and Colimits}
\label{sec:limits-colim}


By a \textbf{limit} in a polycategory \cP we mean a cone of unary morphisms such that
\[ \cP(\Gamma; \Delta,\lim_i A_i) \to \lim_i \cP(\Gamma;\Delta,A_i) \]
is always an isomorphism.
A \textbf{colimit} is a limit in $\cP\op$.
Since limits of modules are pointwise,
\[ \yon_{(\lim_i A_i)} \cong \lim_i (\yon_{A_i}) \qquad\text{and}\qquad
\noy{(\colim_i A_i)} \cong \lim_i (\noy{A_i}),
\]
but in $\modp$ or $\modpj$ we can say nothing about $\yon_{(\colim_i A_i)}$ or $\noy{(\lim_i A_i)}$.
Now let $\pj$ be as in \cref{sec:sheaves}, and \cK a set of {distinguished} limit and colimit cones in \cP.

\begin{rmk}
  We have been ignoring size, but \cP should everywhere be a \emph{small} polycategory, and \cK should be a small set.
  (\cJ is automatically small once \cP is.)
\end{rmk}

\begin{defn}
  A \textbf{$\pjk$-module} is a $\pj$-module that in addition respects the distinguished limits and colimits in \cK, i.e.\ the map
  \[ \cU(\Gamma; \Delta,\lim_i D_i) \to \lim_i \cU(\Gamma;\Delta,D_i) \]
  is an isomorphism for all distinguished limit cones, and dually for colimits.
\end{defn}

The tautological module \cP is a $\pjk$-module, as is any shift of a $\pjk$-module; hence so are $\yon_A$ and $\noy{A}$.
Let $\modpjk\subseteq\modp$ consist of the $\pjk$-modules.


\begin{thm}\label{thm:ssh-yon}\ 
  \begin{enumerate}
  \item $\modpjk$ is a reflective subcategory of $\modp$.\label{item:eis1}
  \item If $\cU\in\modp$ and $\cV\in \modpjk$ then $(\cU\hom\cV)\in\modpjk$.
  \item $\modpjk$ has a closed symmetric monoidal structure such that the reflector is strong monoidal and the inclusion preserves internal-homs.
  \item The embedding $\yon : \Umulti(\cP) \to \modpjk$ preserves distinguished tensor products.
    Dually, the embedding $\noy{} : \Umulti(\cP\op)\to \modpjk$ takes distinguished cotensor products to tensor products.\label{item:eis4}
  \item The embeddings $\yon$ 
    and $\noy{}$ 
    preserve distinguished limits and colimits.\label{item:eis5}
  \end{enumerate}
\end{thm}
\begin{proof}
  Parts~\ref{item:eis1}--\ref{item:eis4} are essentially just like \cref{thm:yoneda-poly,thm:mod-ei}.
  To express respect for a limit as an orthogonality property, we use colimits of modules as in the original~\cite{fk:cts-frs}.
  Note that if the limits and colimits in \cK are finite, then these colimits of modules are still finitely presentable, so that the iterative reflection construction requires only countably many steps and is constructive.

  For~\ref{item:eis5}, it remains to show that $\yon$ preserves distinguished colimits and $\noy{}$ preserves distinguished limits (i.e.\ takes them to colimits).
  To see this, let \cU be a $\pjk$-module; then we have
  \begin{align*}
    \modpjk(\Gamma,\yon_{(\colim_i A_i)}; \cU)
    &\cong \modpjk(\Gamma; \cU[\colim_i A_i\,; \,])\\
    &\cong \modpjk(\Gamma; \lim_i \cU[A_i\,; \,])\\
    &\cong \lim_i \modpjk(\Gamma; \cU[A_i\,;\,])\\
    &\cong \lim_i \modpjk(\Gamma, \yon_{A_i}; \cU)\\
    &\cong \modpjk(\Gamma,\colim_i (\yon_{A_i}); \cU),
  \end{align*}
  using the assumption on $\cU$ in the second step.
  The claim about $\noy{}$ is dual.
\end{proof}

\begin{defn}
  Given $\pjk$, its \textbf{envelope} is the Chu construction
  \[ \envpjk=\chu(\modpjk,\cP).\]
\end{defn}


\begin{thm}\label{thm:sshenv}
  $\envpjk$ is a complete and cocomplete $\ast$-autonomous category, contains \cP as a full sub-polycategory, and the inclusion preserves the distinguished tensor and cotensor products and the distinguished limits and colimits.
\end{thm}
\begin{proof}
  Just like \cref{thm:shenv} except for the final claim, which follows from the formulas for limits and colimits in a Chu construction and \cref{thm:ssh-yon}\ref{item:eis5}:
  \begin{align*}
    \lim_i (A^+_i, A^-_i, e_{i}) &= (\lim_i A^+_i, \colim_i A^-_i, \underline{\hspace{1ex}})\\
    \colim_i (A^+_i, A^-_i, e_{i}) &= (\colim_i A^+_i, \lim_i A^-_i, \underline{\hspace{1ex}})\qedhere
  \end{align*}
\end{proof}

We want to combine this with \cref{thm:smcc-poly} as before, but there is a complication.
For any limit cone in a closed symmetric monoidal category \cC, we might expect
\begin{align*}
  \text{\textquestiondown} \qquad \fsp\cC(\Gamma,\overline\Pi; \Delta, \overline\Sigma, \lim_i A_i)
  &= \cC(\Gamma,\Sigma; \Delta,\Pi, \lim_i A_i)\\
  &\cong \lim_i \cC(\Gamma,\Sigma; \Delta,\Pi, A_i)\\
  &=\lim_i \fsp\cC(\Gamma,\overline\Pi; \Delta, \overline\Sigma, A_i) \qquad ?
\end{align*}
yielding a limit in $\fsp\cC$.
However, the isomorphism in the second line is only valid if $\Delta=\Pi=\emptyset$, whereas to have a limit in $\fsp\cC$ the composite isomorphism must hold for all $\Delta$ and $\Pi$.
Of course, if $\Delta\cup\Pi$ is nonempty, then $\fsp\cC(\Gamma,\overline\Pi; \Delta, \overline\Sigma, \lim_i A_i)$ is empty, as is each $\fsp\cC(\Gamma,\overline\Pi; \Delta, \overline\Sigma, A_i)$ --- but the latter only implies that their limit 
is empty \emph{if the diagram is nonempty!}
And indeed, a terminal object in \cC need not be terminal in $\fsp\cC$: the latter requires $\fsp\cC(\Gamma,\overline\Pi;\Delta,\overline\Sigma,1)=1$ always, but the former only ensures this when $\Pi = \Delta = \emptyset$.
Similar considerations apply for colimits, 
so the best enhancement of \cref{thm:main} we can manage is:

\begin{thm}\label{thm:main2}
  Any closed symmetric monoidal category \cC can be fully embedded in a $\ast$-autonomous category by a strong symmetric monoidal closed functor, which preserves any chosen family of \emph{nonempty} limits and colimits that exist in \cC.\qed
\end{thm}

With other choices of $(\cP,\cJ,\cK)$, we can embed other kinds of structures into $\ast$-autonomous categories as well.
For instance:
\begin{itemize}
\item Any linearly distributive category can be fully embedded in a $\ast$-autonomous category, preserving tensors, cotensors, and any set of colimits that are preserved in each variable by $\ten$ and limits that are preserved in each variable by $\coten$.
\item Any ordinary category can be fully embedded in a $\ast$-autonomous category preserving any set of nonempty limits and colimits.
  (By \cref{rmk:isbell}, this is closely related to the Isbell envelope.)
\end{itemize}

\begin{rmk}
  In the special case when \cC is \emph{cartesian} closed, of course the terminal object $1$ is also the monoidal unit.
  Since the embedding $\cC\into \envcjk$ preserves the monoidal unit, the image of $1\in\cC$ in $\envcjk$ is again the monoidal unit, but it will no longer be terminal.
  And if we include binary products in $\cK$, then this embedding preserves both tensor products and binary cartesian products, which coincide in \cC; hence tensor products and binary cartesian products \emph{of objects in the image of \cC} also coincide in $\envcjk$.
  However, $\envcjk$ is not itself cartesian monoidal: tensor products and cartesian products of objects not in the image of \cC need not coincide.
\end{rmk}

\section{Conservativity}
\label{sec:conservativity}

Since closed symmetric monoidal and $\ast$-autonomous categories model intuitionistic and classical linear logic, respectively, \cref{thm:main2} implies a (1-)conservativity result.

\begin{thm}\label{thm:cons}
  Classical linear logic with $(-)^\bot, \ten, \coten, \unit, \bot, \hom, \&, \oplus$ is conservative over intuitionistic linear logic with $\ten, \unit, \hom, \&, \oplus$ (but not $\mathbf{0},\top$).
\end{thm}
\begin{proof}
  Let \mamll and \imamll denote the given fragments, and let $T$ be a theory in \imamll.
  By imposing an appropriate equivalence relation on proofs from $T$ in \imamll and \mamll respectively, we obtain a closed symmetric monoidal category $\cC_T$ and a $\ast$-autonomous category $\cD_T$, both with binary products and coproducts, and each freely generated by a model of $T$.

  By \cref{thm:main2}, $\cC_T$ embeds fully in $\env{\cC_T}$ preserving all the structure.
  Since $\env{\cC_T}$ is a \mbox{$\ast$-auto}\-nomous model of $T$, this embedding factors up to isomorphism through $\cD_T$.
  Now any sequent in the language of \imamll over $T$ that is provable in \mamll yields a morphism in $\cD_T$ between objects in the image of $\cC_T$.
  Hence such a morphism also exists in $\env{\cC_T}$, and thus also in $\cC_T$.
\end{proof}

Similarly, we can show that \mamll is conservative over any smaller fragment of \imamll, and that full MALL is conservative over any of its fragments.
(IMALL is not a fragment of MALL in this sense, since its judgmental structure is different.)

Perhaps surprisingly, \cref{thm:cons} is almost best possible.
By~\cite{schellinx:syntac-ll}, classical linear logic with $\mathbf{0}$ and $\hom$ is \emph{not} conservative over intuitionistic linear logic with the same connectives: the sequent
\[ C \hom ((\mathbf{0}\hom X) \hom A),\, (C\hom B)\hom \mathbf{0} \vdash A \]
is provable in the former but not the latter.
Thus, not every closed symmetric monoidal category with initial object embeds monoidally into a $\ast$-autonomous category preserving the initial object.

Conservativity of \mamll over \imamll says we can use the former to reason about the latter without changing the theorems that are provable.
Semantically, this means that we can assume a closed symmetric monoidal poset (with binary meets and joins) is $\ast$-autonomous (and hence use the syntax of classical linear logic) without changing the inequalities between objects of the original poset.

To a certain extent, \cref{thm:main2} allows us to similarly use syntaxes for $\ast$-autonomous categories (e.g.\ term syntaxes for linear logic, or graphical calculi such as proof nets and circuit diagrams---see the references cited in \cref{sec:introduction}) to reason about a closed symmetric monoidal category \cC.
For instance, since $\cC\into \env\cC$ is faithful and isomorphism-reflecting, if two morphisms in \cC can be proven equal using $\ast$-autonomous syntax, they were already equal in \cC, and if a morphism in \cC can be proven invertible using $\ast$-autonomous syntax, it was already invertible in \cC.

Since $\cC\into \env\cC$ is full, we can also construct morphisms in \cC in this way; but \textit{a priori} the \emph{particular} morphism in \cC that we obtain might depend on the particular embedding of \cC into its envelope.
And while the envelope seems ``canonical'', it is by no means the \emph{unique} closed monoidal embedding of any given \cC in a $\ast$-autonomous category (for instance, if \cC has a zero object, then $\cC\into \chu(\cC,0)$ suffices).

As noted in the introduction, the most sensible way to avoid this is to show that the \emph{universal} functor $\Phi:\cC\to\cD$ from \cC to a $\ast$-autonomous category is fully faithful, which we call \textbf{2-conservativity}.
In this case, since any fully faithful embedding of \cC in a $\ast$-autonomous category factors through the universal one, all interpretations of $\ast$-autonomous syntax in \cC must coincide.




\cref{thm:main2} implies that $\Phi$ is faithful and conservative, since $\cC\into \env\cC$ factors through it.
We cannot show full-faithfulness of $\Phi$ in the same way, but we can use a general technique introduced by Lafont~\cite{lafont:thesis} that combines Artin gluing along a restricted Yoneda embedding (a.k.a.\ a ``Kripke logical relation''), as generalized to the $\ast$-autonomous case by~\cite{tan:thesis,hs:glue-orth-ll,hasegawa:glueing-cll} using \emph{double gluing}.

\section{Double Gluing}
\label{sec:double-gluing}

The name ``double gluing'' presumably refers to the appearance of \emph{two} ``logical relation'' families, but fortuitously it can also be expressed using \emph{double categories}.
Recall that a (strict) double category is a category internal to $\mathsf{Cat}$; by a \textbf{poly double category} we mean a category internal to \fPoly.
For example, any polycategory \cP induces a poly double category $\dQ\cP$ consisting of the following structure:
\begin{itemize}
\item The objects and the horizontal poly-arrows are those of \cP.
\item The vertical arrows are the unary and co-unary morphisms of \cP.
  (Note that the vertical arrows in any poly double category are only an ordinary category.)
\item The 2-cells are ``commutative squares'' in \cP of the form
  \[
    \begin{tikzcd}
      (A_1,\dots,A_m) \ar[r,"f"]
      \ar[d,shift left=7,"u_m"] \ar[d,shift right=7,"u_1"'] \ar[d,phantom,"\cdots"]  &
      (B_1,\dots,B_n) \ar[d,shift left=7,"v_n"] \ar[d,shift right=7,"v_1"'] \ar[d,phantom,"\cdots"] \\
      (C_1,\dots,C_m) \ar[r,"g"'] &
      (D_1,\dots,D_n)
    \end{tikzcd}
  \]
  i.e.\ the assertion that $g\circ (u_1,\dots,u_m) = (v_1,\dots,v_n)\circ f$.
\end{itemize}

Now since the functor \chu is a right adjoint, it preserves internal categories.
Thus any multicategory \cE with a presheaf \Bbbk has a \textbf{double Chu construction}~\cite{shulman:dialectica}
\[\dChu(\cE,\Bbbk) \coloneqq \chu(\dQ(\cE,\Bbbk)). \]
This is a poly double category described as follows.
\begin{itemize}
\item Its objects and horizontal poly-arrows are those of $\chu(\cE,\Bbbk)$.
\item A vertical arrow $u:A\to B$ is a pair $(u\p : A\p\to B\p, u\m: A\m\to B\m)$ such that $\uB \circ (u\p,u\m) = \uA$.
\item 
  A 2-cell 
  \[
  \begin{tikzcd}
    (A_1,\dots,A_m) \ar[r,"f"]
    \ar[d,shift left=7,"u_m"] \ar[d,shift right=7,"u_1"'] \ar[d,phantom,"\cdots"] \ar[dr,phantom,"\Downarrow\scriptstyle\mu"] &
    (B_1,\dots,B_n) \ar[d,shift left=7,"v_n"] \ar[d,shift right=7,"v_1"'] \ar[d,phantom,"\cdots"] \\
    (C_1,\dots,C_m) \ar[r,"g"'] &
    (D_1,\dots,D_n)
  \end{tikzcd}
  \]
  consists of a family of commuting squares in \sE:
  \[\small\hspace{-2cm}
  \begin{tikzcd}
    (A_1\p, \dots, A_m\p, B_1\m, \dots, \widehat{B_j\m},\dots B_n\m) \ar[r,"f\p_j"{name=f}]
    \ar[d,shift left=22] \ar[d,shift right=22]
    \ar[d,phantom,"{\scriptstyle\cdots (u_1\p,\dots,u_m\p,v_1\m,\dots,\widehat{v_j\m},\dots,v_n\m)\cdots}"]
    &
    B_j\p \ar[d,"v_j\p"] \\
    (C_1\p, \dots, C_m\p, D_1\m, \dots, \widehat{D_j\m},\dots D_n\m) \ar[r,"g\p_j"'{name=g}] &
    D_j\p
  \end{tikzcd}
  \qquad
  \begin{tikzcd}
    (A_1\p , \dots, \widehat{A_i\p},\dots A_m\p , B_1\m, \dots, B_n\m) \ar[r,"f\m_i"{name=f}]
    \ar[d,shift left=22] \ar[d,shift right=22]
    \ar[d,phantom,"{\scriptstyle\cdots(u_1\p , \dots, \widehat{u_i\p},\dots u_m\p , v_1\m, \dots, v_n\m) \cdots}"]
    &
    A_i\m \ar[d,"u_i\m"] \\
    (C_1\p , \dots, \widehat{C_i\p},\dots C_m\p , D_1\m, \dots, D_n\m) \ar[r,"g\m_i"'{name=g}] &
    C_i\m
  \end{tikzcd}\hspace{-2cm}
  \]
  \[\small\hspace{-2cm}
    \begin{tikzcd}
      (A_1\p, \dots, A_m\p, B_1\m,,\dots B_n\m) \ar[r,"\uf"{name=f}]
      \ar[d,shift left=17] \ar[d,shift right=17]
      \ar[d,phantom,"{\scriptstyle\cdots (u_1\p,\dots,u_m\p,v_1\m,\dots,v_n\m)\cdots}"]
      &
      \bbot \ar[d,equals] \\
      (C_1\p, \dots, C_m\p, D_1\m,,\dots D_n\m) \ar[r,"\ug"'{name=g}] &
      \bbot
    \end{tikzcd}
  \]
  (the last follows from the others unless $m=n=0$, when it is the only condition).
\end{itemize}

Now let \cD be a polycategory, and $\Lambda:\cD\to \chu(\cE,\Bbbk)$ a functor, whose action on objects we write as $\Lambda(A) = (L(A),K(A),\lambda_A)$.
On underlying 1-categories, $L$ and $K$ are functors $\cD\to\cE$ and $\cD\to\cE\op$ respectively.

\begin{eg}\label{eg:contraction}
  If \cD and \cE are representable multicategories and $\Bbbk=1$ is terminal, then $\Lambda$ reduces to the input data of~\cite[\S4.2.1]{hs:glue-orth-ll}.
  For applying $\Lambda$ to the universal morphism $(A,B) \to A\ten B$ yields
  \begin{equation*}
    m:L(A)\ten L(B) \to L(A\ten B)\qquad
    k:L(A)\ten K(A\ten B) \to K(B)\qquad
    k':L(B)\ten K(A\ten B) \to K(A)
  \end{equation*}
  of which $k$ and $k'$ determine each other by symmetries.
  Applying $\Lambda$ to units and triple tensors makes $m$ a lax symmetric monoidal structure on $L$ and $k$ a ``contraction'' as in~\cite[\S4.2.1]{hs:glue-orth-ll}, and this determines $\Lambda$.
\end{eg}

\begin{eg}
  By \cref{thm:chu-adjt}, if \cD is a $\ast$-polycategory and $\bbot=1$, a $\ast$-polycategory functor $\Lambda:\cD \to \chu(\cE,1)$ is uniquely determined by $L:\Umulti(\cD)\to \cE$, with $K(A)\coloneqq L(\d A)$.
  If \cD and \cE are representable (hence \cD is $\ast$-autonomous), $L$ is just a lax symmetric monoidal functor; thus $\Lambda$ reduces to the input data of~\cite[\S4.3.1]{hs:glue-orth-ll}.
  Up to isomorphism,
  the same holds when \cD is $\ast$-autonomous without strict duals.
\end{eg}

\begin{defn}
  Let $\psi:\Bbbk_1\to\Bbbk_2$ be presheaf map and $\Lambda:\cD \to \chu(\cE,\Bbbk_2)$.
  The \textbf{double gluing} $\gl(\Lambda,\psi)$ is a comma object in the 2-category of poly double categories and \emph{vertical} transformations:
  \[
    \begin{tikzcd}
      \gl(\Lambda,\psi) \ar[r] \ar[d] \ar[dr,phantom,"\Downarrow"] & \chu(\cE,\Bbbk_1) \ar[d,"\psi"] \\
      \cD \ar[r,"\Lambda"'] & \dChu(\cE,\Bbbk_2).
    \end{tikzcd}
  \]
\end{defn}

Here \cD and $\chu(\cE,\Bbbk_1)$ are regarded as vertically discrete poly double categories.
Hence so is $\gl(\Lambda,\psi)$, i.e.\ it is a plain polycategory.
Its objects consist of
\begin{itemize}
\item An object $A\o\in \cD$.
\item An object $(A\p,A\m, \uA)$ of $\chu(\cE,\Bbbk_1)$.
\item A vertical morphism $(A\p,A\m,\psi\circ\uA) \to (L(A\o),K(A\o),\lambda_A)$ in $\dChu(\cE,\Bbbk_2)$, consisting of $A\p \to L(A\o)$ and $A\m \to K(A\o)$ in \cE such that the composites $(A\p,A\m) \to \Bbbk_1 \xto{\psi} \Bbbk_2$ and $(A\p,A\m) \to (L(A\o),K(A\o)) \xto{\lambda} \Bbbk_2$ agree.
\end{itemize}
Similarly, a morphism $(A_1,\dots,A_m) \to (B_1,\dots,B_n)$ in $\gl(\Lambda,\psi)$ consists of
\begin{itemize}
\item A morphism $f:(A\o_1,\dots,A\o_m) \to (B\o_1,\dots,B\o_n)$ in \cD.
\item A morphism $(f\p_j,f\m_i,\uf)$ in $\chu(\cE,\Bbbk_1)$.
\item The following squares commute:
  \begin{equation}\label{eq:glsq1}
    \begin{tikzcd}
      (\hspace{3mm}A_1\p\hspace{3mm}, \dots, \hspace{2mm}A_m\p\hspace{2mm}, \hspace{2mm}B_1\m\hspace{2mm}, \dots, \hspace{2mm}\widehat{B_j\m}\hspace{2mm},\dots \hspace{3mm}B_n\m\hspace{3mm}) \ar[r,"f\p_j"{name=f}]
      \ar[d,shift left=30] \ar[d,shift right=31]
      \ar[d,shift left=14,phantom,"\cdots"]
      \ar[d,shift right=21,phantom,"\cdots"]
      \ar[d,shift right=13]
      \ar[d,shift right=3]
      &
      B_j\p \ar[d] \\
      (L(A\o_1), \dots, L(A\o_m), K(B\o_1), \dots, \widehat{K(B\o_j)},\dots K(B\o_n)) \ar[r,"\Lambda(f)\p_j"'{name=g}] &
      L(B\o_j)
    \end{tikzcd}
  \end{equation}
  \begin{equation}\label{eq:glsq2}
    \begin{tikzcd}
      (\hspace{3mm}A_1\p\hspace{3mm} , \dots, \hspace{2mm}\widehat{A_i\p}\hspace{2mm},\dots \hspace{2mm}A_m\p\hspace{2mm} ,\hspace{2mm} B_1\m\hspace{2mm}, \dots, \hspace{3mm}B_n\m\hspace{3mm}) \ar[r,"f\m_i"{name=f}]
      \ar[d,shift left=30] \ar[d,shift right=31]
      \ar[d,shift right=14,phantom,"\cdots"]
      \ar[d,shift left=21,phantom,"\cdots"]
      \ar[d,shift left=12]
      \ar[d,shift left=2]
      &
      A_i\m \ar[d] \\
      (L(A\o_1) , \dots, \widehat{L(A\o_i)},\dots L(A\o_m) , K(B\o_1), \dots, K(B\o_n)) \ar[r,"\Lambda(f)\m_i"'{name=g}] &
      K(A\o_i)
    \end{tikzcd}
  \end{equation}
  \begin{equation}\label{eq:glsq3}
    \begin{tikzcd}
      (\hspace{3mm}A_1\p\hspace{3mm}, \dots, \hspace{2mm}A_m\p\hspace{2mm}, \hspace{2mm}B_1\m\hspace{2mm}, \dots, \hspace{3mm}B_n\m\hspace{3mm}) \ar[r,"\uf"{name=f}]
      \ar[d,shift left=23] \ar[d,shift right=23]
      \ar[d,shift left=12,phantom,"\cdots"]
      \ar[d,shift right=14,phantom,"\cdots"]
      \ar[d,shift right=5]
      \ar[d,shift left=5]
      &
      \Bbbk_1 \ar[d,"\psi"] \\
      (L(A\o_1), \dots, L(A\o_m), K(B\o_1), \dots, K(B\o_n)) \ar[r,"\underline{\Lambda(\smash{f})}"'{name=g}] &
      \Bbbk_2
    \end{tikzcd}
  \end{equation}
  (the last follows from the others unless $m=n=0$, when it is the only condition).
\end{itemize}

\begin{thm}\label{thm:double-gluing}
  Suppose \cE is representable and closed with pullbacks, and $\Bbbk_1$ and $\Bbbk_2$ are either both terminal ($\Bbbk_1=\Bbbk_2=1$) or both representable $(\Bbbk_1 = \yon_{\ik_1}$ and $\Bbbk_2 = \yon_{\ik_2})$.
  Then tensor products, cotensor products, duals, and internal-homs exist in $\gl(\Lambda,\psi)$ insofar as they do for the relevant underlying objects in \cD.
  In particular, if \cD is $\ast$-autonomous, so is $\gl(\Lambda,\psi)$.
\end{thm}
\begin{proof}
  When $\Bbbk_1$ and $\Bbbk_2$ are terminal, we use the formulas from~\cite[\S4.2]{hs:glue-orth-ll}.
  When they are representable, we modify the formulas slightly; for tensor products we have
  \[ (A\ten B)\p = A\p \ten B\p \to L(A\o) \ten L(B\o) \to L(A\o\ten B\o) \]
  and the limit of the following diagram (drawn in the middle):
  \[\small
    \begin{tikzcd}[column sep={2.6cm,between origins}]
      && (A\p \mathord{\ten} B\p) \mathord{\hom} \ik_1 & &\\
      & A\p \mathord{\hom} B\m \ar[dl] \ar[ur] & (A\mathord{\ten} B)\m \ar[l,dashed] \ar[r,dashed] \ar[d,dashed] & B\p \mathord{\hom} A\m \ar[dr] \ar[ul]\\
      A\p \mathord{\hom} K(B\o) & L(A\o) \mathord{\hom} K(B\o) \ar[l] & K(A\o\mathord{\ten} B\o) \ar[l] \ar[r] & L(B\o) \mathord{\hom} K(A\o) \ar[r] & B\p \mathord{\hom} K(A\o).
    \end{tikzcd}
  \]
  The unit consists of $\unit\p = \unit \to L(\unit)$ and the pullback
  \[
    \begin{tikzcd}
      \unit\m \ar[rrr] \ar[d] \drpullback & & & \ik_1 \ar[d] \\
      K(\unit) \ar[r,"\cong"'] & \unit\ten K(\unit) \ar[r] & L(\unit) \ten K(\unit) \ar[r,"\lambda"'] & \ik_2\mathrlap.
    \end{tikzcd}
  \]
  The dual of $(A\o,A\p,A\m,\uA)$ is $(\d {(A\o)}, A\m, A\p, \uA\sigma)$, and for the internal-hom we have
  \[\small
    \begin{tikzcd}[column sep={2.6cm,between origins}]
      && (A\p \mathord{\ten} B\m) \mathord{\hom} \ik_1 & &\\
      & A\p \mathord{\hom} B\p \ar[dl] \ar[ur] & (A\mathord{\hom} B)\p \ar[l,dashed] \ar[r,dashed] \ar[d,dashed] & B\m \mathord{\hom} A\m \ar[dr] \ar[ul]\\
      A\p \mathord{\hom} L(B\o) & L(A\o) \mathord{\hom} L(B\o) \ar[l] & L(A\o\mathord{\hom} B\o) \ar[l] \ar[r] & K(B\o) \mathord{\hom} K(A\o) \ar[r] & B\m \mathord{\hom} K(A\o).
    \end{tikzcd}
  \]
  and
  \[ (A\hom B)\m = A\p \ten B\m \to L(A\o) \ten K(B\o) \to K(A\o\hom B\o). \]
  We leave cotensor products to the reader.
\end{proof}

\begin{eg}
  Double gluing is usually described only when $\Bbbk_1=\Bbbk_2=1$, but our more general version appears implicitly in at least one place.
  Specifically, consider \cref{eg:contraction} with $L = \cC(\unit,-)$ and $K = \cC(-,J)$ for some $J$, as in~\cite[\S4.2.2]{hs:glue-orth-ll}.
  Then every object $\Lambda(A) = (LA,KA)$ comes with a pairing $LA\times KA = \cC(\unit,A) \times \cC(A,J) \to \cC(\unit,J)$, which is respected by every morphism in the image of $\Lambda$; thus the codomain of $\Lambda$ lifts to $\chu(\fSet,\ik_2)$ where $\ik_2=\cC(\unit,J)$.
  Now for any $F\subseteq \cC(\unit,J)$ we can take $\psi : \ik_1 = F \into \cC(\unit,J) = \ik_2$, and the resulting $\gl(\Lambda,\psi)$ coincides (modulo a restriction to monomorphisms) with the \emph{slack orthogonality category} of the \emph{focused orthogonality} on $\gl(\Lambda,1)$ determined by $F$ as in~\cite[\S5]{hs:glue-orth-ll}.
\end{eg}

\section{2-Conservativity}
\label{sec:2cons}

Let \cP be a polycategory with \cJ as in \cref{sec:sheaves}, and let $\Phi:\cP\to\cD$ be its universal functor to a \mbox{$\ast$-auto}\-nomous category.
That is, $\Phi$ is a polycategory functor preserving the tensor and cotensor products in \cJ (up to isomorphism), and such that any polycategory functor $\cP \to \cQ$ that preserves \cJ and where \cQ is $\ast$-autonomous factors through $\Phi$, uniquely up to unique isomorphism.
This ``up to isomorphism'' version is categorically ``correct'', and seems necessary since the functor $\Xi$ below does not preserve \cJ strictly.

\begin{thm}\label{thm:2cons}
  The universal $\Phi:\cP\to\cD$ from $(\cP,\cJ)$ to a \mbox{$\ast$-auto}\-nomous category is fully faithful.
\end{thm}
\begin{proof}
  Following Lafont's technique, we will construct the (double) gluing of $\cD$ along the restricted (polycategorical) Yoneda embedding of $\Phi$, and then lift $\Phi$ to a functor $\Xi$ landing in this gluing category that preserves \cJ.
  By the universal property of $\cD$, it will follow that this gluing category has a section, a ``logical relations'' functor that assigns in particular to each morphism in \cD a unique morphism in \cP that maps onto it, showing full-faithfulness of $\Phi$.

  For any \cD-module \cU, we write $\cU[\Phi]$ for the \cP-module with $\cU[\Phi](\Gamma;\Delta) = \cU(\Phi(\Gamma);\Phi(\Delta))$.
  This defines a functor $(-)[\Phi]:\mod\cD \to \modp$; and if \cU respects the tensor and cotensor products of \cD, then $\cU[\Phi]$ is a $(\cP,\cJ)$-module since $\Phi$ preserves \cJ.
  Now we let $\Lambda$ be the composite polycategory functor
  \[ \cD \xto{(\yon,\scriptnoy{},\gm)} \env\cD = \chu(\mod\cD,\cD) \xto{(-)[\Phi]} \chu(\modp,\cD[\Phi]), \]
  where $(\mod\cD,\cD) \to (\modp,\cD[\Phi])$ is induced by $\Phi$.
  Thus on objects we have
  \begin{equation}
    \Lambda(R) = (L(R),K(R),\lambda_R) = (\cD[\,;R][\Phi], \cD[R;\,][\Phi], \gamma_{\Phi R}).\label{eq:fcLK}
  \end{equation}
  Note that $\Lambda$ lands in $\chu(\modpj,\cD[\Phi])$.
  Now $\Phi$ also induces a map $\phi : \cP \to \cD[\Phi]$ in \modpj.
  By \cref{thm:double-gluing},
  the double gluing category
  $\gl(\Lambda,\phi)$ is $\ast$-autonomous.
  
  By 
  the universal property of comma objects, to define a functor $\Xi : \cP \to \gl(\Lambda,\phi)$ it suffices to give the following diagram in poly double categories (most of which are vertically discrete):
  \[
    \begin{tikzcd}
      \cP \ar[r,"{(\yon,\scriptnoy{},\gm)}"] \ar[d,"\Phi"'] \ar[drr,phantom,near end,"\Downarrow"] & \envpj \ar[r,equals] & \chu(\modpj,\cP) \ar[d,"\phi"] \\
      \cD \ar[rr,"\Lambda"'] && \dChu(\modpj,\cD[\Phi])
    \end{tikzcd}
  \]
  The necessary 2-cell has components
  \( \yon_A \to \yon_{\Phi A}[\Phi] \) and \( \noy{A} \to \noy{\Phi A}[\Phi] \), which
  we take to be the action of $\Phi$ on hom-sets.
  Thus $\Xi(A)$ consists of $\Phi A\in \cD$, $(\yon_A, \noy A, \gamma_A) \in \chu(\modpj,\cP)$, and the maps
  \begin{align}
    \yon_A &= \cP[\,; A] \to \cD[\Phi][\,;A] = \cD[\,;\Phi A][\Phi] = \yon_{\Phi A}[\Phi] \label{eq:theta1}\\
    \noy A &= \cP[A;\,] \to \cD[\Phi][A;\,] = \cD[\Phi A;\,][\Phi] = \noy{\Phi A}[\Phi] \label{eq:theta2}
  \end{align}
  We claim that $\Xi$ preserves the tensor and cotensor products in \cJ.
  For the tensors, we use that $\Phi$ and $\yon$ preserve them, and also need to calculate the limit
  \[\small
    \begin{tikzcd}[column sep={2.6cm,between origins}]
      && (\yon_A \mathord{\ten} \yon_B) \mathord{\hom} \cP & &\\
      & \yon_A \mathord{\hom} \noy{B} \ar[dl] \ar[ur] & \bullet \ar[l,dashed] \ar[r,dashed] \ar[d,dashed] & \yon_B \mathord{\hom} \noy{A} \ar[dr] \ar[ul]\\
      \yon_A \mathord{\hom} K(\Phi B) &
      & K(\Phi A\mathord{\ten} \Phi B) \ar[ll] \ar[rr] &
      & \yon_B \mathord{\hom} K(\Phi A).
    \end{tikzcd}
  \]
  Using~\eqref{eq:fcLK} and~\eqref{eq:yoneda3}, this becomes
  \[\small
    \begin{tikzcd}[column sep={2.6cm,between origins}]
      && \cP[A,B;] & &\\
      & \cP[A,B;] \ar[dl] \ar[ur] & \bullet \ar[l,dashed] \ar[r,dashed] \ar[d,dashed] & \cP[A,B;] \ar[dr] \ar[ul]\\
      \cD[\Phi B;][\Phi][A;] \ar[d,"\cong"] &
      & \cD[\Phi A\mathord{\ten} \Phi B;][\Phi] \ar[ll] \ar[rr] \ar[d,"\cong"] &
      & \cD[\Phi A;][\Phi][B;] \ar[d,"\cong"] \\
      \cD[\Phi][A,B;] && \cD[\Phi][A,B;] \ar[ll] \ar[rr] && \cD[\Phi][A,B;]
    \end{tikzcd}
  \]
  whose limit is $\cP[A,B;]\cong \noy{A\ten B}$.
  For a unit, we instead consider the pullback
  \[
    \begin{tikzcd}
      \bullet \ar[r] \ar[d] \drpullback & \cP \ar[d] \\
      \cD[\Phi\unit;][\Phi] \ar[r] 
      & \cD[\Phi].
    \end{tikzcd}
  \]
  But the bottom map is an isomorphism, hence the pullback is isomorphic to $\cP$, which is isomorphic to $\cP[\unit;] = \noy{\unit}$.
  Cotensors are dual.

  Now, since the composite $\cP \xto{\Xi} \gl(\Lambda,\cP) \to \cD$ is equal to $\Phi$, we can extend $\Xi$ by the universal property of $\cD$ to a functor $\Xi':\cD \to \gl(\Lambda,\cP)$.
  \begin{equation*}
    \begin{tikzcd}
      && \gl(\Lambda,\cP) \ar[d] \\
      \cP \ar[r,"\Phi"'] \ar[urr,"\Xi",bend left=5] &
      \cD \ar[ur,dashed,"\Xi'"'] \ar[r,equals] & \cD.
    \end{tikzcd}
  \end{equation*}
  By the universal property of $\gl(\Lambda,\cP)$, this means we have a diagram
  \[
    \begin{tikzcd}
      \cP \ar[d,"\Phi"'] \ar[r,"{(\yon,\scriptnoy{},\gm)}"] \ar[drr,phantom,near end,"\Downarrow\scriptstyle\mu"] & \envpj \ar[r,equals] &
      \chu(\modpj,\cP) \ar[d,"\phi"] \\
      \cD \ar[urr,bend left=5,dashed,"\Xi'",near start] \ar[rr,"\Lambda"'] && \dChu(\modpj,\cD[\Phi])
    \end{tikzcd}
  \]
  where $\mu\Phi$ equals~\eqref{eq:theta1}--\eqref{eq:theta2}.
  This determines the components of $\mu$ on objects $\Phi A$; its naturality on all morphisms in \cD between such objects then entails multiple commuting squares like~\eqref{eq:glsq1}--\eqref{eq:glsq3}.
  Those like~\eqref{eq:glsq3} imply that the quadrilateral involving $\Xi'$ in the following diagram commutes:
  \[\small\hspace{-2cm}
    \begin{tikzcd}[column sep={3.5cm,between origins}]
      \cP(A_1,\dots,A_m;B_1,\dots,B_n) \ar[d,"\Phi"'] \ar[rr,"\cong"] &&
      \modpj(\yon_{A_1},\dots,\yon_{A_m},\noy{B_1},\dots,\noy{B_n},\dots;\cP) \ar[d] \\
      \cD(\Phi A_1,\dots,\Phi A_m;\Phi B_1,\dots,\Phi B_n) \ar[dr] \ar[urr,dashed,"\Xi'"] \ar[rr,dotted,"\cong"'] &&
      \modpj(\yon_{A_1},\dots,\yon_{A_m},\noy{B_1},\dots,\noy{B_n};\cD[\Phi]) \\
      &
      \modpj(\cD[\Phi][;A_1],\dots,\cD[\Phi][;A_m],\cD[\Phi][B_1;],\dots,\cD[\Phi][B_n;];\cD[\Phi]) \ar[ur]
    \end{tikzcd}\hspace{-2cm}
  \]
  Now the two horizontal arrows are isomorphisms by \cref{thm:yoneda}.
  Thus, by the 2-out-of-6 property for isomorphisms, the left-hand vertical map $\Phi$
  is also an isomorphism; i.e.\ $\Phi$ is fully faithful.
\end{proof}

\begin{rmk}
  Most of the proof would work using the traditional $\gl(\Lambda,1)$ instead of our $\gl(\Lambda,\phi)$.
  When showing that $\Xi$ preserves \cJ, we would omit $(\yon_A \ten \yon_B) \hom \cP$ from the diagram, changing the limit to the kernel-pair of $\cP[A,B;] \to \cD[\Phi][A,B;]$; but since \cref{thm:main2} implies $\Phi$ is faithful, this morphism is monic, so its kernel pair is just its domain.
  Also, the squares like~\eqref{eq:glsq3} would carry no information; but we could use~\eqref{eq:glsq1}--\eqref{eq:glsq2} instead as long as \cJ includes a unit or counit, so that we can ignore the $m=n=0$ case.
  However, our proof seems cleaner and easier to to generalize.
\end{rmk}

\begin{cor}\label{thm:2cons-csmc}
  Let \cC be a closed symmetric monoidal category and $\Phi:\cC\to\cD$ its universal functor to a $\ast$-autonomous one.
  Then $\Phi$ is fully faithful.
\end{cor}
\begin{proof}
  A closed symmetric monoidal functor from \cC to a $\ast$-autonomous category factors essentially uniquely through $(\fsp\cC,\cJ)$.
  Thus $\Phi':\fsp\cC\to\cD$ is the universal functor from $(\fsp\cC,\cJ)$ to a $\ast$-autonomous category; now apply \cref{thm:2cons}.
\end{proof}

To emphasize how this proof deals with the internal-homs, we check explicitly that the induced $\Xi:\cC\to \gl(\Lambda,\phi)$ preserves them.
We must inspect the limit
\[\small
  \begin{tikzcd}[column sep={2.6cm,between origins}]
    && (\yon_A \mathord{\ten} \noy{B}) \mathord{\hom} \fsp\cC & &\\
    & \yon_A \mathord{\hom} \yon_B \ar[dl] \ar[ur] & \bullet \ar[l,dashed] \ar[r,dashed] \ar[d,dashed] & \noy{B} \mathord{\hom} \noy{A} \ar[dr] \ar[ul]\\
    \yon_A \mathord{\hom} \cD[;B][\Phi] &
    & \cD[;A\mathord{\hom} B][\Phi] \ar[ll] \ar[rr] &
    & \noy{B} \mathord{\hom} \cD[A;][\Phi].
  \end{tikzcd}
\]
which by~\eqref{eq:fcLK} and~\eqref{eq:yoneda3} becomes
\[\small
  \begin{tikzcd}[column sep={2.6cm,between origins}]
    && \fsp\cC[A;B] & &\\
    & \fsp\cC[A;B] \ar[dl] \ar[ur] & \bullet \ar[l,dashed] \ar[r,dashed] \ar[d,dashed] & \fsp\cC[A;B] \ar[dr] \ar[ul]\\
    \cD[;\Phi B][\Phi][A;] \ar[d,"\cong"] &
    & \cD[;\Phi(A\mathord{\hom} B)][\Phi] \ar[ll] \ar[rr]\ar[d,"\cong"] &
    & \cD[\Phi A;][\Phi][;B]\ar[d,"\cong"]\\
    \cD[\Phi][A;B] && \cD[\Phi][A;B] \ar[ll] \ar[rr] && \cD[\Phi][A;B],
  \end{tikzcd}
\]
whose limit is $\fsp\cC[A;B] \cong \fsp\cC[;A\mathord{\hom} B] = \yon_{A\hom B}$.
We also need to check that
\[ \yon_A \ten \noy{B} \to \cD[;\Phi A \hom \Phi B][\Phi]
\quad\text{is isomorphic to}\quad
\noy{A\hom B} \to \cD[;\Phi(A\hom B)][\Phi]. \]
But by \cref{thm:sh-emb}, we have
\( \noy{A\hom B} = \noy{\Abar \coten B} \cong \noy{\Abar} \ten \noy{B} \cong \yon_A \ten \noy{B}. \)
This is the crucial point: the polycategorical dual Yoneda embedding $\noy{}$ maps internal-homs to tensor products.

\begin{rmk}\label{rmk:mistake}
  It is claimed in~\cite{hs:glue-orth-ll,hasegawa:glueing-cll} that \cref{thm:2cons-csmc} can be proven using double gluing into an ordinary presheaf category, but it seems that this does not work.
  An ordinary Yoneda embedding has no dual like $\noy{}$, so the ``dual parts'' of $\Xi$ have to be chosen ``tautologically''; but then $\Xi$ fails to preserve the internal-homs.
  In the notation of~\cite[\S4.5]{hasegawa:glueing-cll}, $(\dP_A \hom \dP_B)_t(X)$ is the set of morphisms $\dI X\to \dI A \ten (\dI B)^\bot$ in $\dC_1$ that factor as $(\dI f\ten g) \circ \dI h$ for some $h\in \dC_0(X,Y\ten Z)$, $f\in \dC_0(Y,A)$, and $g\in \dC_1(\dI Z, (\dI B)^\bot)$, whereas $(\dP_{A\hom B})_t(X)$ is the set of \emph{all} morphisms in $\dC_1(\dI X, \dI A \ten (\dI B)^\bot)$.
  There seems no reason why every such morphism should factor in that way.
\end{rmk}

\cref{thm:2cons} also specializes to other polycategorical structures.
For instance, we have the following result (shown in~\cite{bcst:natded-coh-wkdistrib} by cut-elimination).

\begin{cor}
  The universal functor from any linearly distributive category to a $\ast$-autonomous one is fully faithful.
\end{cor}
\begin{proof}
  By~\cite{cs:wkdistrib}, a linearly distributive category can be regarded as a representable polycategory.
  Now apply \cref{thm:2cons} with all tensors and cotensors in \cJ.
\end{proof}

We can also include a family \cK of limits and colimits by double gluing with $\modpjk$ instead.
The formulas in~\cite[Proposition 31]{hs:glue-orth-ll} for products and coproducts in double gluing categories still work, as do similar ones for other limits and colimits, and the functor $\Xi$ preserves them.
Of course, as in \cref{sec:limits-colim}, only nonempty limits and colimits in a multicategory \cC induce polycategorical ones in $\fsp\cC$.

\section{Acknowledgments}
\label{sec:acknowledgments}

This material is based on research sponsored by The United States Air Force Research Laboratory under agreement number FA9550-15-1-0053.  The U.S. Government is authorized to reproduce and distribute reprints for Governmental purposes notwithstanding any copyright notation thereon.  The views and conclusions contained herein are those of the author and should not be interpreted as necessarily representing the official policies or endorsements, either expressed or implied, of the United States Air Force Research Laboratory, the U.S. Government, or Carnegie Mellon University.

I would like to thank Max New, Sam Staton, Kirk Sturtz, and Todd Trimble for useful conversations, as well as the referees of the conference version for useful suggestions.

\bibliographystyle{eptcsalpha}
\bibliography{eptcs-references}

\end{document}